\documentclass[12pt,a4paper,dvipsnames]{amsart}

\usepackage{amssymb,amsmath,amsthm}
\usepackage{enumerate,ushort,bbm,titletoc,a4}
\usepackage{tikz}
\usetikzlibrary{matrix,arrows,decorations.pathmorphing}
\usepackage[titletoc]{appendix}
\usepackage[pagebackref,colorlinks,linkcolor=BrickRed,citecolor=OliveGreen,urlcolor=blue,hypertexnames=true]{hyperref}

\DeclareMathOperator{\dom}{dom}
\DeclareMathOperator{\End}{End}
\DeclareMathOperator{\rank}{rank}
\DeclareMathOperator{\Char}{char}
\DeclareMathOperator{\id}{id}

\DeclareMathOperator{\co}{co}

\DeclareMathOperator{\GL}{GL}
\DeclareMathOperator{\Rel}{Rel}
\DeclareMathOperator{\Jac}{Jac}
\renewcommand{\span}{{\rm span}}

\renewcommand{\emptyset}{\varnothing}
\renewcommand{\setminus}{\smallsetminus}
\renewcommand{\subset}{\subseteq}

\usepackage[left=2.5cm, right=2.5cm, top=3.2cm, bottom=3.1cm]{geometry}

\def\kk{\mathbbm k}

\def\bes{\begin{equation*}}
\def\ees{\end{equation*}}

\def\beq{\begin{equation} }
\def\eeq{\end{equation} }

\def\ben{\begin{enumerate} }
\def\een{\end{enumerate} }
\def\benum{\begin{enumerate} }
\def\eenum{\end{enumerate} }

\def\x{\ushort X}

\def\ua{\ushort A}

\def\ub{\ushort B}
\def\pp{\ushort P}
\def\ss{\intercal}
\def\xs{\x^\ss}
\def\y{\ushort Y}

\newcommand{\cx}{[\x]}
\newcommand{\Langle}{\!\mathop{<}\!}
\newcommand{\Rangle}{\!\mathop{>}}
\newcommand{\ax}{\Langle\x\Rangle}
\newcommand{\axy}{\Langle\x,\y\Rangle}
\newcommand{\axs}{\Langle\x,\xs\Rangle}

\newcommand{\fr}[1]{\kk\Langle #1 \Rangle}
\newcommand{\frc}[1]{\C\Langle #1 \Rangle}

\def\bmat{\left[\begin{array}{ccccccccccccccc} }
\def\emat{\end{array}\right]}

\def\bmat{\begin{bmatrix}}
\def\emat{\end{bmatrix}}
\def\beq{\begin{equation}}
\def\eeq{\end{equation}}

\def\barr{\begin{array}}
\def\earr{\end{array}}

\makeatletter
\def\moverlay{\mathpalette\mov@rlay}
\def\mov@rlay#1#2{\leavevmode\vtop{%
    \baselineskip\z@skip \lineskiplimit-\maxdimen
    \ialign{\hfil$#1##$\hfil\cr#2\crcr}}}
\makeatother

\makeatletter
\def\Ddots{\mathinner{\mkern1mu\raise\p@
\vbox{\kern7\p@\hbox{.}}\mkern2mu
\raise4\p@\hbox{.}\mkern2mu\raise7\p@\hbox{.}\mkern1mu}}
\makeatother

\newcommand{\plangle}{\moverlay{(\cr<}}
\newcommand{\prangle}{\moverlay{)\cr>}}
\def\rax{\plangle\x\prangle}

\newcommand{\ff}[1]{\kk\plangle #1 \prangle}

\def\N{\mathbb N}

\def\cA{ {\mathcal A} }

\def\cC{ {\mathcal C} }

\def\cH{ {\mathcal H} }
\def\cI{ {\mathcal I} }

\def\cM{{\mathcal M}}
\def\cN{ {\mathcal N} }

\def\cP{{\mathcal P}}

\def\cS{{\mathcal S} }
\def\cT{{\mathcal T}}

\def\cW{{\mathcal W}}
\def\cU{\mathcal U}
\def\cX{\mathcal X}

\def\cZ{ {\mathcal Z} }

\def\j1tog{ j= 1, \ldots, g }

\def\cI{{\mathcal I}}

\def\L0t{\ (L_0 \otimes I_n ) \ }

\def\PLINE1{\beta}

\newcommand{\C}{{\mathbb C}}

\newcommand{\R}{{\mathbb R}}

\DeclareMathOperator{\Hom}{Hom}

\DeclareMathOperator{\im}{im}
\newcommand{\der}{\mathrm d}

\newcommand{\bb}{\mathbf{b}}
\newcommand{\cc}{\mathbf{c}}

\newcommand{\uu}{\mathbf{u}}
\newcommand{\vv}{\mathbf{v}}
\newcommand{\fW}{\mathfrak{W}}
\newcommand{\fY}{\mathfrak{Y}}
\newcommand{\fy}{\mathfrak{y}}
\def\y{\ushort Y}
\newcommand{\gp}{\Langle \y \Rangle}
\newcommand{\agp}{\cA\gp}
\newcommand{\gs}{\Langle\!\!\Langle \y \Rangle\!\!\!\Rangle}
\newcommand{\ags}{\cA\gs}

\newtheorem{thm}{Theorem}[section]

\newtheorem{cor}[thm]{Corollary}
\newtheorem{lem}[thm]{Lemma}

\newtheorem{prop}[thm]{Proposition}
\theoremstyle{definition}

\theoremstyle{remark}
\newtheorem{rem}[thm]{Remark}

\numberwithin{equation}{section}

\newtheorem{exa}[thm]{Example}
\newtheorem{defn}[thm]{Definition}

\newtheorem{nota}[thm]{Notation}

\newcounter{Inc}

\begin{document}

\setcounter{tocdepth}{3}
\contentsmargin{2.55em} 
\dottedcontents{section}[3.8em]{}{2.3em}{.4pc} 
\dottedcontents{subsection}[6.1em]{}{3.2em}{.4pc}
\dottedcontents{subsubsection}[8.4em]{}{4.1em}{.4pc}

\makeatletter
\newcommand{\mycontentsbox}{%
{\centerline{NOT FOR PUBLICATION}
\addtolength{\parskip}{-4.3pt}
\tableofcontents}}
\def\enddoc@text{\ifx\@empty\@translators \else\@settranslators\fi
\ifx\@empty\addresses \else\@setaddresses\fi
\newpage\mycontentsbox\newpage\printindex}
\makeatother

\setcounter{page}{1}

\title[Null- and Positivstellens\"atze]{Null- and Positivstellens\"atze for\\[1mm] rationally resolvable ideals}

\author[Igor Klep]{Igor Klep${}^1$}
\address{Igor Klep, The University of Auckland, Department of Mathematics}
\email{igor.klep@auckland.ac.nz}
\thanks{${}^1$Supported by the Marsden Fund Council of the Royal Society of New Zealand. Partially supported by the Slovenian Research Agency grants P1-0222, L1-4292 and L1-6722. Part of this research was done while the author was on leave from the University of Maribor.}

\author[Victor Vinnikov]{Victor Vinnikov}
\address{Victor Vinnikov, 
Ben-Gurion University of the Negev, Department of Mathematics}
    \email{vinnikov@math.bgu.ac.il}

\author[Jurij Vol\v{c}i\v{c}]{Jurij Vol\v{c}i\v{c}${}^2$}
\address{Jurij Vol\v{c}i\v{c}, The University of Auckland, Department of Mathematics}
    \email{jurij.volcic@auckland.ac.nz}
\thanks{${}^2$Research supported by The University of Auckland Doctoral Scholarship.}

\date{\today}

\begin{abstract}
Hilbert's Nullstellensatz characterizes polynomials that vanish on the vanishing set of an ideal in $\C\cx$. In the free algebra $\C\ax$ the vanishing 
set of a two-sided ideal $\cI$ is defined in a dimension-free way using 
images in finite-dimensional representations of $\C\ax\!/\cI$.  
In this article Nullstellens\"atze for a simple but important 
class of ideals in the free algebra -- called tentatively {\em rationally resolvable} here -- are presented. An ideal is rationally resolvable if its defining relations can be eliminated by expressing some of the $\x$ variables  using noncommutative rational functions in the remaining variables. 
Whether such an ideal $\cI$ satisfies the Nullstellensatz
is intimately related to embeddability of $\C\ax\!/\cI$ into (free) skew fields.
These notions are also extended to free algebras with involution.
For instance, it is  proved that a polynomial vanishes on all tuples of spherical isometries iff it is a member of the two-sided ideal $\cI$ generated by $1-\sum_j X_j^\ss X_j$. This is then applied to free real algebraic geometry: polynomials positive semidefinite on  spherical isometries 
are sums of Hermitian squares modulo $\cI$. Similar results are
obtained for nc unitary groups.
\end{abstract}
\maketitle

\section{Introduction}

In algebraic geometry 
Hilbert's Nullstellensatz is a classical result 
characterizing polynomials vanishing on an algebraic set:

\begin{thm}[Hilbert's Nullstellensatz]\label{thm:hilbert}
Let $f,h_1,\ldots,h_s\in\C\cx$ and
$$
Z:=\{a\in\C^g\mid h_1(a)=\cdots=h_s(a)=0\}.
$$
If $f|_Z=0$, then
for some $r\in\N$, $f^r$ belongs to the ideal $(h_1,\ldots, h_s)$.
\end{thm}
Due to its importance it has been generalized and extended in 
many different directions. For instance, 
there are  noncommutative versions due to Amitsur \cite{Ami57}, Bergman \cite{HM}, and
Helton et al. \cite{HMP2,CHMN}.
Here our main interest is in
\emph{free noncommutative Nullstellens\"atze} 
describing vanishing in free algebras. That is, given a two-sided ideal
$\cI$ in a free algebra $\kk\ax$, we consider polynomials $f$ vanishing under
all finite-dimensional representations of $\kk\ax\!/\cI$.
If each such $f$ is in $\cI$, then we say that $\cI$ has the Nullstellensatz property. Hence this is a noncommutative analog of a radical ideal in classical algebraic geometry.

In this paper we focus on an important
class of ideals we call rationally resolvable. These are ideals in which all generators can
be eliminated by solving for some of the variables in terms of noncommutative rational functions
in the remaining variables (see Definition \ref{def:RR} for a precise statement). Noncommutative rational functions are elements of a free skew field $\kk\rax$, i.e., the field with ``the least'' rational relations between its generators $X_j$. If for a rationally resolvable ideal $\cI$ the quotient ring $\kk\ax\!/\cI$ admits a
``nice'' embedding in a (free) skew field, then $\cI$ has the Nullstellensatz
property; the rigorous formulation is given in Theorem \ref{thm1} and Proposition \ref{prop2}. For instance, we obtain the following.

\begin{thm}\label{t:intro1}
Let $\cI\subset\kk\ax\!/\cI$ be a (formally) rationally resolvable ideal. If $\kk\ax\!/\cI$ is a Sylvester domain (e.g., a free ideal ring), then $\cI$ satisfies the Nullstellensatz property.
\end{thm}

See Theorem \ref{thm1}(a) and Proposition \ref{prop2}(a) for the proof. Subsection \ref{s:exa} gives examples illustrating the strength of our results in Section \ref{sec2}. In particular, we show the following.

\begin{cor}\label{c:intro1}
Let $f\in\fr{\x,\y}$.
\begin{enumerate}[\rm (a)]
\item If $f(\ua,\ub)=0$ for all $n\in\N$ and $(\ua,\ub)\in M_n(\kk)^{2g}$ such that $A_jB_j=I_n$ for $j=1,\dots,g$, then $f\in (1-X_1Y_1,1-Y_1X_1,\dots,1-X_gY_g,1-Y_gX_g)$.
\item If $f(\ua,\ub)=0$ for all $n\in\N$ and $(\ua,\ub)\in M_n(\kk)^{2g}$ with $A_1B_1+\dots+A_gB_g=I_n$, then $f\in (1-X_1Y_1-\cdots-X_gY_g)$.
\item If $f(\ua,\ub)=0$ for all $n\in\N$ and $(\ua,\ub)\in M_n(\kk)^{2g^2}$ such that $(A_{ij})_{ij}(B_{ij})_{ij}=I_{gn}$, then
$$f\in \left( \delta_{ij}-\sum_{k=1}^g X_{ik}Y_{kj},\delta_{ij}-\sum_{k=1}^g Y_{ik}X_{kj}\colon 1\le i,j\le g\right),$$
where $\delta_{ij}$ denotes Kronecker's delta.
\end{enumerate}
\end{cor}

See Corollaries \ref{c:exfrr1} and \ref{c:exfrr2} for the proofs. To obtain size bounds needed on the dimensions of the finite-dimensional representations of $\kk\ax\!/\cI$ for these Nullstellens\"atze, we employ systems
theory realizations for noncommutative rational functions; see \cite[Chapters 1 and 2]{BR} or \cite{BGM, HMV, KVV09}.
Our rational functions do not admit scalar regular points in general, so we present the
necessary modifications of the classical theory to handle matrix centers in Section \ref{sec3}. 
As a side product we obtain size bounds needed to test for rational identities, see Theorem \ref{t:RIbound}. This machinery is then applied to Nullstellens\"atze in Subsection \ref{s:bounds}: for a noncommutative polynomial $f$ and a rationally resolvable ideal $\cI$ we give 
explicit bounds on the dimension of the finite-dimensional representations
of $\kk\ax\!/\cI$ needed to test whether $f$ vanishes under all these representations (Theorem \ref{t:rrbounds}).

Section \ref{sec4} applies our results to $*$-ideals in 
free algebras with involution. We show that 
the $*$-ideals corresponding to unitaries and spherical isometries \cite{HMP1} satisfy the 
Nullstellensatz property (Theorem \ref{thm:null2}), and give
in Theorem \ref{thm:null3} a Nullstellensatz
for noncommutative unitary groups 
\cite{Bro,Wo}. These results can be summarized as follows.

\begin{cor}\label{c:intro2}
Let $f\in\C\axs$.
\begin{enumerate}[\rm (a)]
\item If $f(\ua,\ua^*)=0$ for all $n\in\N$ and $\ua\in M_n(\kk)^g$ where $A_j$ are unitary, then $f\in (1-X_1X_1^{\ss},1-X_1^{\ss}X_1,\dots, 1-X_gX_g^{\ss},1-X_g^{\ss}X_g)$.
\item If $f(\ua,\ua^*)=0$ for all $n\in\N$ and $\ua\in M_n(\kk)^g$ with $A_1A_1^{\ss}+\dots+A_gA_g^{\ss}=I_n$, then $f\in (1-X_1X_1^{\ss}-\cdots-X_gX_g^{\ss})$.
\item If $f(\ua,\ua^*)=0$ for all $n\in\N$ and $\ua\in M_n(\kk)^{g^2}$ such that $(A_{ij})_{ij}$ is a unitary matrix, then
$$f\in \left( \delta_{ij}-\sum_{k=1}^g X_{ik}X_{kj}^{\ss},\delta_{ij}-\sum_{k=1}^g X_{ik}^{\ss}X_{kj}\colon 1\le i,j\le g\right).$$
\end{enumerate}
\end{cor}

As before, these results are effective (i.e., we obtain concrete size bounds). To extend our involution-free statements to $*$-ideals in free
algebras with involution, we use (real) algebraic geometry, cf.~Subsection \ref{subs:real}. The paper concludes in Subsection \ref{subsec:pos} with Positivstellens\"atze for 
a few selected examples of rationally resolvable ideals. For example, the following result solves a problem from \cite{HMP1} on spherical isometries, i.e., tuples of matrices $(A_1,\dots,A_g)$ satisfying $A_1^*A_1+\cdots+A_g^*A_g=I$.

\begin{thm}\label{t:intro2}
If $f\in\C\axs$ of degree $d-1$ is nonnegative on all spherical isometries of size $(2g+1)^d$, then
$$f=\sum_jp_j^{\ss}p_j+q$$
for some $p_j\in\C\axs$ of degrees at most $d$ and $q\in (1-X_1^{\ss}X_1 \cdots-X_g^{\ss}X_g)$.
\end{thm}

Theorem \ref{t:intro2} is proved as Corollary \ref {c:pos2} in Subsection \ref{subsec:pos}.

\subsection*{Acknowledgments} 
The authors thank Warren Dicks and Pere Ara
for their help with Corollary \ref{c:exfrr2}.

\section{Rationally resolvable ideals and Nullstellens\"atze} \label{sec2}

In this section we introduce two notions of rationally resolvable ideals and identify important classes and examples of those satisfying a Nullstellensatz. Loosely speaking, an ideal is rationally resolvable if from each of its generators one variable can be expressed as a noncommutative rational function of other variables; see Definition \ref{def:RR} for a precise statement. The main result here is Theorem \ref{thm1}, and interesting examples are presented in Subsection \ref{s:exa}.

\subsection{Notation and terminology}

\subsubsection{Words and nc polynomials}\label{subsec:NCpoly}
Given a field $\kk$, and positive integers $n,g$,
  let $M_n(\kk)$ denote the $n \times n$ matrices
  with entries from $\kk$ and $M_n(\kk)^g$
  the set of $g$-tuples of such $n\times n$ matrices.
For simplicity, we shall always assume $\Char \kk=0$.

We write $\ax$ for the monoid freely
generated by $\x=\{X_1,\ldots, X_d\}$, i.e., $\ax$ consists of {\bf words} in the $g$
noncommuting
letters $X_{1},\ldots,X_g$
(including the empty word $\emptyset$ which plays the role of the identity $1$). With $|w|\in \N\cup\{0\}$ we denote the length of word $w\in\ax$. Let $\kk\ax$ denote the associative
$\kk$-algebra freely generated by $\x$, i.e., the elements of $\kk\ax$
are polynomials in the noncommuting variables $\x$ with coefficients
in $\kk$. Its elements are called {\bf (nc) polynomials}.
An element of the form $aw$ where $0\neq a\in \R$ and
$w\in\ax$ is called a {\bf monomial} and $a$ its
{\bf coefficient}. 
Sometimes we also use $Y_j$ to denote noncommuting variables.

\subsubsection{Polynomial evaluations}

If $p\in\kk\ax$ is an nc polynomial and $\ua\in M_n(\kk)^g$,
the evaluation $p(\ua)\in M_n(\kk)$ is defined by simply replacing $X_{i}$ by $A_{i}$. For example, if $p(x)=3 X_{1} X_{2}-X_1^2$, 
then
$$
p\left( \begin{bmatrix}
1&1 \\
-1 & 0
\end{bmatrix},
\begin{bmatrix}
1&0\\
2&-1
\end{bmatrix}
\right)=
3  \begin{bmatrix}
1&1 \\
-1 & 0
\end{bmatrix}
\,
\begin{bmatrix}
1&0\\
2&-1
\end{bmatrix}
- 
\begin{bmatrix}
1&1 \\
-1 & 0
\end{bmatrix}^2=
\begin{bmatrix}
9 & -4\\
-2 & 1
\end{bmatrix}.
$$
In other words, evaluations of nc polynomials at
$\ua\in M_n(\kk)^g$ are representations
$${\rm ev}_{\ua}:\kk\ax\to M_n(\kk), \quad
p\mapsto p(\ua).
$$

\subsubsection{Ideals}

Let $S=\{f_1,\ldots,f_s\}\subseteq \kk\ax$. The {\bf two-sided ideal} generated by $S$ is
\beq\label{eq:2i}
\cI(S):= \Big\{ \sum_{i=1}^n \sum_{j=1}^s g_{ij} f_j h_{ij} \mid n\in\N,\, g_{ij},h_{ij}\in\kk\ax\Big\}.
\eeq
It is the smallest subset of $\kk\ax$ containing $S$ and closed
under addition and multiplication (from left and right) with
elements from $\kk\ax$.

\subsubsection{Zero sets}

There are many noncommutative generalizations of Hilbert’s Nullstellensatz. 
The ones that concern us here replace point evaluations $\C\cx\to\C$ 
by some class of representations of the noncommutative algebra that we are dealing with. For the case of the free algebra $\kk\ax$
the most reasonable class of representations,
so far, seems to be the class of all finite dimensional representations. 
In other words, we replace evaluations of commutative polynomials on points in $\C^g$ by evaluations of nc polynomials on $g$-tuples of $n\times n$ 
matrices over $\kk$, for all $n\in\N$.

There is (at least) one other choice to be made in the noncommutative setting: we have to decide whether we are dealing with one-sided ideals or 
with two-sided ideals giving rise to different types of vanishing.
We refer the reader to \cite{BK} for a more extensive overview
of noncommutative Nullstellens\"atze.
Here we focus on (strong) vanishing and hence on two-sided ideals.

To a (two-sided) ideal $\cI\subset\kk\ax$ we associate its {\bf zero set}
\beq
Z(\cI):= \bigcup_{n\in\N} \{ \ua\in M_n(\kk)^g \mid
\forall g\in\cI: \,g(\ua)=0 \}.
\eeq
In contrast with the classical case (cf.~Theorem \ref{thm:hilbert}),
it may very well happen that $Z(\cI)=\emptyset$ for
a proper ideal $\cI\subsetneq\kk\ax$. For instance, take
$\cW=(X_1X_2-X_2X_1-1)$.

\subsection{The free skew field}

The universal skew field of fractions of the free algebra
$\kk\ax$ is called the {\bf free skew field}
and denoted by 
$\kk\rax$.
We call its elements (nc) rational functions.
Let us describe in a bit more detail in a language
suitable for our investigation how they are obtained.
We start with nc polynomials in $\kk\ax$, add
their formal inverses, allow addition and multiplication,
and then repeat this procedure.
This gives rise to nc rational expressions.
We emphasize that an expression includes the order in
which it is composed and no two distinct expressions are
identified, e.g., $(X_1)+(-X_1)$, $(-1)+(((X_1)^{-1})(X_1))$, and
$0$ are different nc rational expressions. 
 
 Evaluation of polynomials naturally extends to rational expressions. 
 If $r$ is a rational expression in $\x$ and $\ua\in M_n(\kk)^g$, 
 then $r(\ua)$ is defined - in the obvious way - as long as any inverses appearing actually exist.
 Generally, a 
nc rational expression $r$ can be evaluated on a
$g$-tuple $A$ of $n \times n$ matrices in its domain of
regularity, $\dom{r}$, which is defined as the set of all
$g$-tuples of square matrices of all sizes such that all the
inverses involved in the calculation of $r(\ua)$ exist. 
From here on we assume that all rational expressions under consideration have nonempty domains.

  Two rational
 expressions $r_1$ and $r_2$ are equivalent
 if $r_1(\ua)=r_2(\ua)$ at any $\ua\in\dom r_1\cap \dom r_2$.
 Then an equivalence class of rational expressions is a rational function.
The set $\kk\rax$ of all rational functions is a skew field. It has the following universal property: if $D$ is a skew field, then every homomorphism $\phi:\kk\ax\to D$ extends to a {\bf local homomorphism} from $\kk\rax$ to $D$, i.e., $\phi$ extends to a homomorphism $\varphi: K\to D$ for some subring $\kk\ax\subseteq K\subseteq \kk\rax$ such that for every $u\in K$, $\varphi(u)\neq0$ implies $u^{-1}\in K$. For a comprehensive study of (free) skew fields we refer to Cohn \cite{Coh1,Coh2}.

A rational expression is of {\bf height $h$} if the maximal number of nested inverses in it is $h$. The {\bf height} of a rational function is then defined as the minimum of heights of all the rational expressions representing it. We let $\ff{\x}_h\subset \ff{\x}$ denotes the subring of all rational functions whose height is at most $h$. If $r$ is a tuple of rational functions, then $h(r)$ denotes the maximum of heights of its components.

\def\cW{\mathcal W}

\subsection{Rationally resolvable ideals}

\begin{nota}\label{RR:notation}
We first introduce some additional notation. Fix a partition of the variables
\begin{equation}\label{decom}
\x=\x'\cup \x''=\{X_1',\dots,X_{g'}'\}\cup \{X_1'',\dots,X_{g''}''\},
\end{equation}
hereafter called the {\bf decomposition} of $\x$, and a tuple $r=(r_1,\dots,r_{g''})$ of rational functions $r_j\in\ff{\x'}$, to which we assign the following objects. Let $\dom r\subset \bigcup_n M_n(\kk)^{g'}$
be the common domain of regularity of the $r_j$ and
$$\Gamma(r)=\{(\ua,r(\ua))\mid \ua\in \dom r\}\subseteq \bigcup_n M_n(\kk)^g$$
the {\bf graph} of $r$. Let $R_r$ be the subring of $\ff{\x}$ generated by $\ff{\x'}_{h(r)}$ and $\fr{\x}$. In particular, $R_r$ contains $\x''$ and all $r_j$. Finally, let $\cI_r$ be the ideal in $R_r$ generated by the set $\{\x''_j-r_j(\x')\}_j$.
\end{nota}

\begin{defn}\label{def:RR}
Let $\cI$ be an ideal in $\fr{\x}$.

\begin{enumerate}[\rm(1)]
\item $\cI$ is {\bf formally rationally resolvable (frr)} with respect to a decomposition of $\x$ and a tuple $r$ as in Notation \ref{RR:notation} if 
\begin{enumerate}[\rm(a)]
\item $\cI\cap \fr{\x'}=0$; and 
\item $\cI$ generates $\cI_r$ as an ideal in $R_r$.
\end{enumerate}
\item $\cI$ is {\bf geometrically rationally resolvable (grr)} with respect to a decomposition of $\x$ and a tuple $r$ as in Notation \ref{RR:notation} if 
\begin{enumerate}[\rm(a)]
\item $\Gamma(r)\subseteq Z(\cI)$; and 
\item every polynomial in $\fr{\x}$, which vanishes on $\Gamma(r)$, also vanishes on $Z(\cI)$.
\end{enumerate}
\end{enumerate}
In both cases we call $r$ the {\bf rational resolvent} of $\cI$.
\end{defn}

From a geometric perspective the grr property is more
appealing, however frr is easier to handle algebraically.

\def\cT{\mathcal T}
\def\cS{\mathcal S}

\subsection{A Nullstellensatz} \label{s:main}

Theorem \ref{thm1} below is the basic version of our Nullstellensatz.

\begin{defn}
An ideal $\cI\subset\kk\ax$ is said to have the
{\bf Nullstellensatz property} if
for each $f\in\kk\ax$,
\beq
f\in\cI \quad\Leftrightarrow\quad f|_{Z(\cI)}=0.
\eeq
\end{defn}

In classical algebraic  geometry this coincides with being
a radical ideal by Hilbert's Nullstellensatz.

\begin{prop}\label{prop1}
Assume the setup is as in Notation {\rm\ref{RR:notation}}. Then:
\begin{enumerate}[\rm(a)]
\item There exists an isomorphism $\alpha:R_r\!/\cI_r \to \ff{\x'}_{h(r)}$;
\item A polynomial $p\in\fr{\x}$ vanishes on $\Gamma(r)$ if and only if $p\in \cI_r\cap \fr{\x}$;
\item Assume the ideal $\cI$ of $\fr{\x}$ is grr with rational resolvent $r$. Then $\cI$ satisfies the Nullstellensatz property if and only if $\cI=\cI_r\cap \fr{\x}$.
\end{enumerate}
\end{prop}

\begin{proof} (i) The isomorphism is given by $\x'\mapsto \x'$ and $\x''\mapsto r(\x')$.
\\
(ii) Let $p\in\fr{\x}$ be arbitrary. By \cite[Theorem 16]{Ami}, $p(\ua,r(\ua))=0$ for all $\ua\in\dom r$ if and only if $q=p(\x',r(\x'))$ equals 0 in $\ff{\x'}$. Since the height of $q$ is at most $h(r)$, the former is equivalent to $p\in \cI_r$ by (i), so the claim follows.
\\
(iii) By the definition, $\cI$ is grr and has the Nullstellensatz property if and only if
$$\cI=\left\{p\in\fr{\x}\mid p|_{\Gamma(r)}=0\right\},$$
but the latter equals $\cI_r\cap \fr{\x}$ by (ii).
\end{proof}

Let $S$ be a subring of $R_1$ and $R_2$. 
The
{\bf coproduct of $R_1$ and $R_2$ over $S$} is a ring $R_1*_SR_2$ with homomorphisms $\vartheta_1:R_1\to R_1*_SR_2$ and $\vartheta_2:R_2\to R_1*_SR_2$ satisfying $\vartheta_1|_S=\vartheta_2|_S$ such that for every other ring $U$ with homomorphisms $\vartheta_1':R_1\to U$ and $\vartheta_2':R_2\to U$ satisfying $\vartheta_1'|_S=\vartheta_2'|_S$ there exists a unique homomorphism $\vartheta:R_1*_SR_2\to U$ such that $\vartheta_1'=\vartheta\circ\vartheta_1$ and $\vartheta_2'=\vartheta\circ\vartheta_2$. In other words, $R_1*_SR_2$ is the categorical pushout of $R_1$ and $R_2$ over $S$. We say that $R_1*_SR_2$ is {\bf faithful} if the canonical maps $R_1\to R_1*_SR_2$ and $R_2\to R_1*_SR_2$ are one-to-one.

\begin{thm}\label{thm1}
Let $\cI$ be an ideal of $\fr{\x}$.
\begin{enumerate}[\rm(a)]
\item If $\cI$ is frr with rational resolvent $r$ and the coproduct of $\fr{\x}\!/\cI$ and $\ff{\x'}_{h(r)}$ over $\fr{\x'}$ is faithful, then $\cI$ is grr and satisfies the Nullstellensatz property.
\item If $\cI$ is grr and satisfies the Nullstellensatz property, then $\fr{\x}\!/\cI$ embeds into a free skew field.
\end{enumerate}
\end{thm}

\begin{proof}(a) Assume there exists a commutative diagram
\begin{center}
\begin{tikzpicture}
\matrix(m)[matrix of math nodes,
row sep=2.6em, column sep=2.8em,
text height=1.5ex, text depth=0.25ex]
{\fr{\x'} & \fr{\x}\!/\cI \\
\ff{\x'}_{h(r)} & \cC \\};
\path[right hook->,font=\scriptsize,>=angle 90]
(m-1-1) edge (m-1-2)
edge (m-2-1);
\path[->,font=\scriptsize,>=angle 90]
(m-1-2) edge node[auto] {$\phi_1$} (m-2-2)
(m-2-1) edge node[below] {$\phi_0$} (m-2-2);
\end{tikzpicture}
\end{center}
where
$$\cC=\fr{\x}\!/\cI*_{\fr{\x'}}\ff{\x'}_{h(r)}$$
is the coproduct. Since $\cI\subseteq \cI_r$, the inclusion $\fr{\x}\subseteq R_r$ induces a homomorphism $\beta:\fr{\x}\!/\cI\to R_r\!/\cI_r$. Since the sets $\x'$ and $\x''$ are disjoint, $R_r$ is the coproduct of $\fr{\x}$ and $\ff{\x'}_{h(r)}$ over $\fr{\x'}$, thus $\phi_0$ and the composite of the quotient map $\fr{\x}\to\fr{\x}\!/\cI$ and $\phi_1$ induce a homomorphism $R_r\to \cC$ whose kernel includes $\cI$ and thus also $\cI_r$. Therefore there exists a homomorphism $\phi_2:R_r\!/\cI_r \to \cC$ such that the diagram
\begin{center}
\begin{tikzpicture}
\matrix(m)[matrix of math nodes,
row sep=2.6em, column sep=2.8em,
text height=1.5ex, text depth=0.25ex]
{\fr{\x}\!/\cI & \cC\\
R_r\!/\cI_r&  \\};
\path[->,font=\scriptsize,>=angle 90]
(m-1-1) edge node[auto] {$\phi_1$} (m-1-2)
edge node[left] {$\beta$} (m-2-1)
(m-2-1) edge node[below] {$\phi_2$} (m-1-2);
\end{tikzpicture}
\end{center}
commutes. Since $\phi_1$ is injective by assumption, $\beta$ is also injective, and so $\cI=\cI_r\cap\fr{\x}$. Therefore $\cI$ is grr and has the Nullstellensatz property by Proposition \ref{prop1}(iii).

(b) Let $r$ be the rational resolvent of $\cI$. Consider the mapping
\begin{equation}\label{Phi}
\Phi: \fr{\x}\to \ff{\x'},\quad p\mapsto p(\x',r(\x')).
\end{equation}
This homomorphism is a composition of a quotient map, $\beta$ from (A), $\alpha$ from Proposition \ref{prop1}(i) and an inclusion. But $\alpha$ is injective by the assumption and Proposition \ref{prop1}(iii), so we have $\ker \Phi= \cI$ and therefore obtain an embedding $\fr{\x}\!/\cI \hookrightarrow \ff{\x'}$.
\end{proof}

The second assumption in Theorem \ref{thm1}(a) might be somewhat difficult to check, so we 
present in Proposition \ref{prop2} two alternative sufficient conditions for the conclusion of Theorem \ref{thm1}(a) to hold that are easier to verify.

Proposition \ref{prop2} and its proof rely heavily upon the theory of skew fields as presented in \cite{Coh1}, so we recall a few definitions. A $n\times n$ matrix $A$ over a ring $R$ is {\bf full} if $A=BC$ for $n\times m$ matrix $B$ and $m\times n$ matrix $C$ implies $m\ge n$. A homomorphism $\varphi:R\to S$ is {\bf honest} if it maps full matrices over $R$ to full matrices $S$ when applied entry-wise. A ring $R$ is a {\bf free ideal ring}, or {\bf fir}, if every left (right) ideal in $R$ is a free left (right) $R$-module of unique rank; see \cite[Section 2.2]{Coh2} and \cite[Proposition 5.5.1]{Coh2}. Typical examples are free algebras and free group algebras. The most important property of a fir for our purpose is that it honestly embeds into its universal skew field of fractions. However, being fir is quite a strong condition; for example, $\kk[x,y]$ is not a fir since $(x,y)$ is not a free left $\kk[x,y]$-module. A bit wider class of rings with honest embeddings into universal skew fields of fractions are Sylvester domains \cite[Theorem 4.5.8]{Coh0}; see also \cite[Section 5.5]{Coh2} and \cite{DS} for an extensive study.

\begin{prop}\label{prop2}
Let $\cI$ be frr with rational resolvent $r$. Then the coproduct of $\fr{\x}\!/\cI$ and $\ff{\x'}_{h(r)}$ over $\fr{\x'}$ is faithful if
\begin{enumerate}[\rm(a)]
\item $\fr{\x}\!/\cI$ is a fir, or more generally, a Sylvester domain; or
\item $h(r)\le 1$ and $\fr{\x}\!/\cI$ embeds into a skew field.
\end{enumerate}
\end{prop}

\begin{proof}
Observe that the coproduct of $R_1$ and $R_2$ over $S$ is faithful if $R_1$ and $R_2$ embed into some ring in a way such that the embeddings agree on $S$.

(a) We claim that the inclusion $\fr{\x'}\hookrightarrow \fr{\x}\!/\cI$ is honest. Indeed, assume that a $n\times n$ matrix $A$ over $\fr{\x'}$ has a factorization $A=B_0C_0$, where $B_0$ is $n\times m$ matrix and $C_0$ is $m\times n$ matrix over $\fr{\x}\!/\cI$. Therefore there exist matrices $B_1$ and $C_1$ over $\fr{\x}$ of the same sizes as $B_0$ and $C_0$, respectively, such that $A-B_1C_1$ is a matrix over $\cI$. If we apply the homomorphism $\Phi$ from \eqref{Phi} in the proof of Theorem \ref{thm1}(b) on $B_1$ and $C_1$ entry-wise, we get $A=B_2C_2$ for matrices $B_2$ and $C_2$ over $\ff{\x'}$ of appropriate sizes. Since $\fr{\x'}$ is a Sylvester domain with the universal skew field of fractions $\ff{\x'}$, the inclusion $\fr{\x'}\subset \ff{\x'}$ is honest, so there exist matrices $B_3$ and $C_3$ over $\fr{\x'}$ of sizes $n\times m$ and $m\times n$, respectively, such that $A=B_3C_3$. Therefore the claim holds. Thus $\ff{\x'}$ embeds into the universal skew field of fractions $D$ of $\fr{\x}\!/\cI$ by \cite[Theorem 4.5.10]{Coh1}, so
\begin{center}
\begin{tikzpicture}
\matrix(m)[matrix of math nodes,
row sep=2.6em, column sep=2.8em,
text height=1.5ex, text depth=0.25ex]
{\fr{\x'} & \fr{\x}\!/\cI \\
\ff{\x'} & D \\};
\path[right hook->,font=\scriptsize,>=angle 90]
(m-1-1) edge (m-1-2)
edge (m-2-1)
(m-1-2) edge (m-2-2)
(m-2-1) edge (m-2-2);
\end{tikzpicture}
\end{center}
is a desired diagram.

(b) If $\fr{\x}\!/\cI$ embeds into a skew field $D$, then $\fr{\x'}$ also embeds into $D$. Since $\ff{\x'}$ is the universal skew field of fractions of $\fr{\x'}$, there exists a local homomorphism from $\ff{\x'}$ to a skew subfield of $D$ generated by the image of $\fr{\x'}$. Therefore $\ff{\x'}_1$ also embeds into $D$ and the diagram
\begin{center}
\begin{tikzpicture}
\matrix(m)[matrix of math nodes,
row sep=2.6em, column sep=2.8em,
text height=1.5ex, text depth=0.25ex]
{\fr{\x'} & \fr{\x}\!/\cI \\
\ff{\x'}_1 & D \\};
\path[right hook->,font=\scriptsize,>=angle 90]
(m-1-1) edge (m-1-2)
edge (m-2-1)
(m-1-2) edge (m-2-2)
(m-2-1) edge (m-2-2);
\end{tikzpicture}
\end{center}
commutes.
\end{proof}

\subsection{Examples and counterexamples}\label{s:exa}
In this subsection we present examples illustrating the strength of our results. We start with a simple family of ideals satisfying the assumptions of Theorem \ref{thm1}(a). Thus these ideals are all rationally resolvable in both senses and satisfy the Nullstellensatz property.

\begin{exa}
For some $i_1,\dots,i_m,j_1,\dots, j_n\in\N$ with $i_1\neq j_1$ and $i_m\neq j_n$ consider
$$\cI=(X_{i_1}\cdots X_{i_m}-X_{j_1}\cdots X_{j_n})\subset\fr{\x}.$$
As a side product of a result by Lewin and Lewin in \cite[Theorem 3]{LL} on embedding of the group algebra of a torsion-free one-relator group into a skew field, $\fr{\x}\!/\cI$ embeds into a skew field by \cite[Corollary 6.3]{LL}. If at least one symbol appears exactly once in the given relation, $\cI$ is also frr, so $\cI$ satisfies the Nullstellensatz property by Proposition \ref{prop2}(b) and Theorem \ref{thm1}(a).

For example, if $\cI=(X_1X_2X_3-X_3X_1^2)$, then we choose the decomposition $\x=\{X_1,X_3\}\cup\{X_2\}$ and the resolvent $r=X_1^{-1}X_3X_1^2X_3^{-1}$. Note that $R_r$ is then the ring generated by $\kk\ax$ and $f^{-1}$ for nonzero $f\in\kk\!\Langle X_1,X_3\Rangle$, and $\cI$ generates the ideal $\cI_r=(X_2-X_1^{-1}X_3X_1^2X_3^{-1})$ in $R_r$.
\end{exa}

We continue with two families of ideals satisfying the Nullstellensatz property that will be revisited in Section \ref{sec4} where we discuss noncommutative unitary groups and spherical isometries.

\subsubsection{Towards nc unitary groups}

For $1\le \ell\le n$ let $X_\ell=(X_{ij}^{(\ell)})_{ij}$ and $Y_\ell=(Y_{ij}^{(\ell)})_{ij}$ be $g_\ell\times g_\ell$ matrices of free noncommuting symbols. Moreover, let $\Rel_\ell$ be the set of entries of the matrices $X_\ell Y_\ell-I_{g_\ell}$ and $Y_\ell X_\ell-I_{g_\ell}$. Consider the ideal
$$\cU'=(\Rel_1\cup\cdots\cup \Rel_n)$$
in the ring
$$\fr{\x,\y}=\fr{X_{ij}^{(\ell)},Y_{ij}^{(\ell)}\colon 1\le \ell\le n,\ 1\le i,j\le g_\ell}.$$

\begin{cor}\label{c:exfrr1}
	The ideal $\cU'$ satisfies the Nullstellensatz property.
\end{cor}

\begin{proof}
The quotient $\fr{\x,\y}\!/\cU'$ is a fir by \cite[Theorem 6.1]{Ber1} and \cite[Theorem 5.3.9]{Coh1} since
$$\fr{\x,\y}\!/\cU' = \fr{\x^{(1)},Y^{(1)}}\!/(\Rel_1)*_k\cdots *_k
\fr{\x^{(n)},Y^{(n)}}\!/(\Rel_n).$$
Thus it is a Sylvester domain by \cite[Proposition 4.5.5]{Coh1}. Also, $\cU'$ is frr, since  $Y_\ell$'s entries can be expressed as rational functions of $X_\ell$'s entries from the defining equations. This can be done by recursive application of the blockwise inversion formula (see e.g. \cite[Subsection 0.7.3]{HJ}),
$$\begin{bmatrix}A&B \\ C&D\end{bmatrix}^{-1}=
\begin{bmatrix}
(A-BD^{-1}C)^{-1} & A^{-1}B(CA^{-1}B-D)^{-1} \\
(CA^{-1}B-D)^{-1}CA^{-1} & (D-CA^{-1}B)^{-1}
\end{bmatrix}$$
on matrices $X_\ell$. Therefore $\cU'$ satisfies the condition of Proposition \ref{prop2}(a).
\end{proof}

\subsubsection{Towards spherical isometries}

Let $\cS'=(X_1Y_1+\cdots+X_gY_g-1)\subset \fr{\x,\y}$ for variables $\x=\{X_1,\dots,X_g\}$ and $\y=\{Y_1,\dots,Y_g\}$.

\begin{cor}\label{c:exfrr2}
	The ideal $\cS'$ satisfies the Nullstellensatz property.
\end{cor}

\begin{proof}
Since $\cS'$ is frr with rational resolvent $r=X_1^{-1}(1-X_2Y_2-\cdots-X_gY_g)$, by Proposition \ref{prop2}(b) it suffices to prove that $\fr{\x,\y}\!/\cS'$ embeds into a skew field. Let
$$R := \kk \Langle X_{ij}, Y_{ij} \mid 1\le i,j \le g,\,(X_{ij})(Y_{ij})=(Y_{ij})(X_{ij})=I_g\Rangle.$$
As seen in Corollary \ref{c:exfrr1} this ring is a fir and therefore embeddable into a skew field.
There is a normal form in $R$ consisting of nc polynomials without terms containing 
$X_{i1}Y_{1j}$ or $Y_{i1}X_{1j}$ (cf.~\cite{Coh0}).
It is easy to see that there is a normal form in the ring 
$\kk\axy\!/\cS'$  consisting of polynomials without $X_1 Y_1$. Now map 
\[
\kk\axy\!/\cS'  \to R,\quad X_i \mapsto X_{2i}, 
\quad  Y_i \mapsto Y_{i2}.
\]
This is an embedding at the normal form level, 
hence $\kk\axy\!/\cS'$ embeds in $R$.
\end{proof}

Note that $\fr{\x,\y}\!/\cS'$ is a hereditary ring, $(g-1)$-fir but not $g$-fir by \cite[Theorem 6.1]{Ber1}, hence it is not a Sylvester domain by \cite[Proposition 11]{DS}.

\subsubsection{Counterexamples}

Lastly, we list a few examples which show that the assumptions of Theorem \ref{thm1} cannot be weakened.

\begin{exa}\label{counterex} \hfill \\[-4ex]
\ben[\rm(1)]

\item Even if an ideal $\cI$ is rationally resolvable in both senses, this does not imply the Nullstellensatz property or that $\fr{\x}\!/\cI$ embeds into a skew field; an easy counterexample is $(1-XY)\subset \fr{X,Y}$.

\item The Weyl algebra is an Ore domain and therefore embeddable into a skew field, but its defining ideal  $(XY-YX-1)\subset \fr{X,Y}$ does not satisfy the Nullstellensatz property.

\item The ring $\fr{X,Y}\!/(XY)$ has zero divisors and therefore cannot embed into a skew field, but $(XY)$ has the Nullstellensatz property. This follows from the fact that for every $m,n\in \N_0$, not both zero, there exist matrices $A$ and $B$ such that
$$B^m A^n\neq0,\quad AB=0,\quad B^{m+1}=0=A^{n+1}.$$
Concretely, one can choose the $(m+n+1)\times (m+n+1)$ matrices
$$A=\sum_{i=1}^n E_{i,i+1},\quad B=\sum_{i=n+2}^{n+m}E_{i,i+1}+E_{m+n+1,1}$$
where $E_{i,j}$ are the standard matrix units.

\item The property frr does not imply grr. The ideal $\cI=(X-XYX)\subset\fr{X,Y}$ is frr with rational resolvent $r=X^{-1}$. Assume that $\cI$ is grr. Then obviously exactly one of the symbols $X$ and $Y$ belongs to the first set of the decomposition \eqref{decom}, for example $X$ (the other case is treated similarly). Thus $(X-XYX)$ is grr with rational resolvent $s=p(X)q(X)^{-1}$, where $p$ and $q$ are coprime univariate polynomials. Consider the polynomial $q(X)Y-p(X)$. It obviously equals 0 on $\Gamma(s)$, but it does not vanish in $(0,0)$  if $p(0)\neq 0$ or in $(1,1)$ if $p(0)=0$. This is a contradiction since  these two points belong to the zero set of $(X-XYX)$.

\item The Nullstellensatz property together with grr does not imply frr. Consider the ideal
$$(ZW-WZ,Z^3-W^2)\subset \fr{X,Y,Z,W}$$
and rational functions $r_Z=(XY^{-1}X)^2$ and $r_W=(XY^{-1}X)^3$. Then
$$\cI_r\cap \fr{X,Y,Z,W}=(ZW-WZ,Z^3-W^2).$$
The inclusion $\supseteq$ is trivial and the inclusion $\subseteq$ 
follows from the fact that 
the set of words of the form
$$m_0Z^{e_1}W^{f_1}m_1\cdots Z^{e_l}W^{f_l}m_l,$$
where $f_i\in\{0,1\}$ and $m_i$ are words in $X$ and $Y$, is linearly independent over $\kk$ in $R_r\!/\cI_r$. Thus the given ideal is grr and satisfies the Nullstellensatz property by Proposition \ref{prop1}(c). However, it is not frr. Indeed, otherwise at least one of $Z,W$ is in the second set of the decomposition \eqref{decom}, say $W$, and then
\begin{equation}\label{eq1}
W-s=\sum_i a_i(ZW-WZ)b_i+\sum_i c_i(Z^3-W^2)d_i
\end{equation}
holds for some $a_i,b_i,c_i,d_i\in F\Langle W\Rangle$ and $s\in F$, where $F=\ff{X,Y,Z}$. Since $F[W]$ is a homomorphic image of $F\Langle W\Rangle$. Then the equation \eqref{eq1}  implies $W-s$ is in the ideal of $F[W]$ generated by a polynomial of degree 2, namely the image of $W^2-Z^3$, which is a contradiction.
\een
\end{exa}

\section{\except{toc}{Realization theory for noncommutative rational functions and bounds for the Nullstellensatz} \for{toc}{Realization theory and bounds for the Nullstellensatz}} \label{sec3}

In this section we give size bounds needed to check the vanishing property $f|_{Z(\cI)}=0$ for a grr ideal $\cI$. More precisely, we present a concrete bound $N$, depending on $f$ and the rational resolvent of $\cI$, such that $f|_{Z(\cI)}=0$ is equivalent to
$f|_{Z(\cI)\cap M_N(\kk)^g}=0$; see Theorem \ref{t:rrbounds}.

Let $\cI$ be grr with rational resolvent $r$. Recall that
$$f|_{Z(\cI)}=0 \iff f|_{\Gamma(r)}=0,$$
which is furthermore equivalent to $f(\x',r(\x'))$ being a rational identity. Therefore we are interested in providing bounds for testing whether a rational expression is a rational identity. This can be achieved through realization theory for rational expressions (see \cite{BR, KVV09, BGM, HMV} for realization theory of rational expressions defined in a scalar point and \cite{CR} for realizations over infinite-dimensional skew fields). Here we present its aspects that are relevant for the task at hand; a more thorough discussion of this subject will be given elsewhere \cite{Vol}. As we shall need power series expansions about non-scalar points, we start by introducing generalized polynomials (which will form homogeneous components of these power series) in Subsection \ref{s:gp}. Subsection \ref{s:gs} then presents the general realization theory, and its application to the bounds for the Nullstellensatz property are in Subsection \ref{s:bounds}.
As an auxiliary result we present size bounds 
needed to test whether a rational expression is a rational identity,
see Theorem \ref{t:RIbound}.

Throughout this section let $\y=\{Y_1,\dots,Y_g\}$ be a set of freely noncommuting letters and $\cA=M_m(\kk)$.

\subsection{Generalized polynomials}\label{s:gp}

The elements of the free product
$$\agp=M_m(\kk) *_{\kk}\kk\gp$$
are called \textbf{generalized (nc) polynomials} over $M_m(\kk)$. They can be evaluated in 
$M_{ms}(\kk)$ via embedding $a\mapsto a\otimes I_s$ of $M_m(\kk)$ into $M_{ms}(\kk)$ and we have
\begin{equation}\label{e:mrfunctor1}
\Hom_{M_m}(M_m(\kk)\gp,M_{ms}(\kk))\cong \Hom(\fW_m(\kk\gp),M_s(\kk)),
\end{equation}
where $\fW_m$ denotes \textit{the matrix reduction functor} as in \cite[Section 1.7]{Coh1}. For a free algebra we have
$$\fW_m(\kk\gp)=\kk\langle \fY\rangle,$$
where
$$\fY=\{\fy_{ij}^{(k)}\colon 1\le i,j\le m,\ 1\le k\le g\}$$
is a set of independent freely noncommuting letters. The isomorphism \eqref{e:mrfunctor1} follows from the isomorphism
\begin{equation}\label{e:mrfunctor2}
M_m(\kk)\gp\to M_m(\kk\langle \fY\rangle),\quad 
E_{i\imath}Y_k E_{\jmath j}\mapsto \fy_{\imath\jmath}^{(k)}\cdot E_{ij},
\end{equation}
where $E_{ij}$ are the standard matrix units in $M_m(\kk)$.

\begin{prop}\label{p:GPI}
If $f\in \agp$ is of degree $h$ and vanishes on matrices of size 
$m\lceil\tfrac{h+1}{2}\rceil$, then $f=0$.
\end{prop}

\begin{proof}
Since the isomorphism \eqref{e:mrfunctor2} preserves polynomial degrees, the proposition follows from \eqref{e:mrfunctor1} and a well-known fact that there are no nonzero polynomial identities on $M_s(\kk)$ of degree less than $2s$ 
(see e.g. \cite[Lemma 1.4.3]{Row}).
\end{proof}

For $Y_i\in\y$ let $\cA^{Y_i}$ denote the $\cA$-bimodule in $\agp$ generated by $Y_i$, i.e. $\cA^{Y_i}=\sum \cA Y_i \cA$, and more generally, $\cA^w= \cA^{w_1}\cdots \cA^{w_{|w|}}$ for $w\in \gp$ (here we set $\cA^1=\cA$). Note that $\cA^w$ and $\cA^{\otimes |w|}$ are isomorphic as $\cA$-bimodules; in particular, as a $\cA$-bimodule, $\cA^w$ does not depend on the letters in $w$, but just on the length of $w$.

\begin{lem}\label{l:difvar}
Let $w=Y_1\cdots Y_k$. If $f\in\cA^w$ vanishes on $\cA$, then $f=0$.
\end{lem}

\begin{proof}
We prove the claim by induction on $k$. Assume $f\in\cA^{Y_1}$ vanishes on $\cA$ and consider the homomorphism of $\kk$-algebras
$$\phi:\cA\otimes_{\kk}\cA^{\text{op}}\to \End_{\kk}(A),\quad a\otimes b\mapsto L_a R_b,$$
where $L_a$ and $R_b$ are multiplications by $a$ on the left and by $b$ on the right, respectively. It is a classical result (see e.g. \cite[Theorem 3.1]{La}) that the $\kk$-algebra 
$\cA\otimes\cA^{op}$ is simple. Therefore $\phi$ is injective. Since $f$ can be considered as an element of $\cA\otimes\cA^{op}$ and its evaluation then corresponds to $\phi(f)$, we have $f=0$.

Now assume that statement holds for $k$. Suppose $f\in\cA^{Y_1}\cdots\cA^{Y_{k+1}}$ vanishes on $\cA$. We can write it as
$$f=\sum_i f_iY_{k+1} a_i,$$
where $a_i\in \cA$ are $\kk$-linearly independent and $f_i\in \cA^{Y_1}\cdots\cA^{Y_k}$. Let 
$b_1,\dots b_k\in \cA$ be arbitrary and $\tilde{f}=f(b_1,\dots,b_k,y_{k+1})$. By the basis of induction we have $\tilde{f}=0$. Since $a_i$ are $\kk$-linearly independent, we have $f_i(b_1,\dots,b_k)=0$ for all $i$. Since $b_j$ were arbitrary, we have $f_i=0$ by the induction hypothesis and therefore $f=0$.
\end{proof}

\subsection{Generalized series}\label{s:gs}

The completion of $\agp$ with respect to the $(\y)$-adic topology is the algebra of \textbf{generalized formal series} over $\cA$ and is denoted by $\ags$.
We refer the reader to \cite{KVV,Voi04,Voi10,AM,HKM,Pas} for analytic approaches to free function theory.

 If a series $S\in \ags$ is written as
\begin{equation}\label{exp}
S=\sum_{w\in \gp} \sum_{i=1}^{n_w}a_{w,i}^{(0)}w_1a_{w,i}^{(1)}w_2\cdots w_{|w|}a_{w,i}^{(|w|)},
\end{equation}
then let
\begin{equation}
[S,w]=\sum_i a_{w,i}^{(0)}w_1a_{w,i}^{(1)}w_2\cdots w_{|w|}a_{w,i}^{(|w|)}.
\end{equation}
This is a well-defined element of $\agp$ even though the expansion \eqref{exp} is not uniquely determined. Note that the homogeneous components of $S$, i.e., $\sum_{|w|=h}[S,w]$ for fixed $h\in\N_0$, also belong to $\agp$.

In this exposition, we treat generalized series in a purely algebraic way. However, if $\kk$ is a field of real or complex numbers, one can also consider matrix norms and therefore study the convergence of generalized series in the norm topology; see \cite[Section 8.2]{KVV} for details.

As in the classical setting, a series $S\in\ags$ is invertible if and only if $[S,1]$ is invertible in $\cA$; in that case we have $[S^{-1},1]=a^{-1}$ and
\begin{equation}
\label{e:inv}
[S^{-1},w]=-\sum_{\substack{uv=w,\\ v\neq w}} a^{-1}[S,u][S^{-1},v]
\end{equation}
for $|w|>0$.

A series $S$ is \textbf{recognizable} if for some $n\in\N$ there exist $\cc\in \cA^{1\times n}$, $\bb\in \cA^{n\times 1}$ and 
$A^{Y_i}\in (\cA^{Y_i})^{n\times n}$ for $Y_i\in \y$ such that $[S,w]=\cc A^w\bb$ for all $w\in \gp$, where the notation
$$A^w = A^{w_1}\cdots A^{w_{|w|}} \in(\cA^w)^{n\times n}$$
is used. In this case $(\cc,A,\bb)$ is called a \emph{linear representation of dimension $n$} of $S$. Observe that one can also write
$$S=\sum_w \cc A^w \bb=\cc\left(\sum_w A^w \right)\bb
=\cc\left(I_n-\sum_{i=1}^g A^{Y_i}\right)^{-1}\bb.$$

The following theorem shows that the set of recognizable series is closed under basic arithmetic operations.

\begin{thm}\label{t:arithmetic}
For $i\in\{1,2\}$ let $S_i$ be a recognizable series with representation $(\cc_i,A_i,\bb_i)$ of dimension $n_i$ and $S$ be an invertible recognizable series with representation $(\cc,A,\bb)$ of dimension $n$. Then:
\begin{enumerate}[\rm(1)]
\item $1$ is recognizable with representation $(1,0,1)$;
\item $Y_i+a$ for $a\in\cA$ is recognizable with representation
\begin{equation}\label{a1}
\left(
\begin{pmatrix}1& a\end{pmatrix},
\begin{pmatrix}0&\delta_{ij}Y_j\\ 0&0\end{pmatrix},
\begin{pmatrix}0\\ 1\end{pmatrix}
\right)
\end{equation}
of dimension 2, where $\delta_{ij}$ is the Kronecker's delta;
\item $S_1+a S_2$ is recognizable with a representation
\begin{equation}\label{a2}
\left(
\begin{pmatrix}\cc_1 & a\cc_2\end{pmatrix},
\begin{pmatrix}A_1&0\\0&A_2\end{pmatrix},
\begin{pmatrix}\bb_1\\ \bb_2\end{pmatrix}
\right)
\end{equation}
of dimension $n_1+n_2$;

\item $S_1S_2$ is recognizable with a representation
\begin{equation}\label{a3}
\left(
\begin{pmatrix}\cc_1 & \cc_1\bb_1\cc_2\end{pmatrix},
\begin{pmatrix}A_1&A_1\bb_1\cc_2 \\0&A_2\end{pmatrix},
\begin{pmatrix}0\\ \bb_2\end{pmatrix}
\right)
\end{equation}
of dimension $n_1+n_2$; 

\item $S^{-1}$ is recognizable with a representation
\begin{equation}\label{a4}
\left(
\begin{pmatrix}-a^{-1}\cc & a^{-1}\end{pmatrix},
\begin{pmatrix}A(I-\bb a^{-1}\cc)&A\bb a^{-1}\\0&0\end{pmatrix},
\begin{pmatrix}0\\ 1\end{pmatrix}
\right)
\end{equation}
of dimension $n+1$, where $a=[S,1]$.
\end{enumerate}

\end{thm}

\begin{proof}
(1) and (2) are trivial.

(3) This is clear since
$$[S_1+aS_2,w]=[S_1,w]+a[S_2,w]=\cc_1 A_1^w\bb_1+a\cc_2 A_2^w\bb_2=
\begin{pmatrix}\cc_1&a\cc_2\end{pmatrix}\cdot
\begin{pmatrix}A_1^w&0\\0&A_2^w\end{pmatrix}\cdot
\begin{pmatrix}\bb_1\\ \bb_2\end{pmatrix}.$$

(4) If $w=w_1\cdots w_\ell$, let $M_j=A_1^{w_j}$, $N_j=A_2^{w_j}$ and $Q=\bb_1\cc_2$. Since
\[ \begin{split}
[S_1S_2,w]
&=\sum_{uv=w}[S_1,u][S_2,w]
=\cc_1\bb_1\cc_2 A_2^w\bb_2+\cc_1\left(\sum_{uv=w,|u|>0} A_1^u \bb_1\cc_2 A_2^v \right)\bb_2 \\
&=(\cc_1 Q)N_1\cdots N_\ell c_2+\cc_1\left(
\sum_{k=1}^\ell M_1\cdots M_k Q N_{N+1}\cdots N_\ell \right)\bb_2,
\end{split} \]
it is enough to prove the equality
$$\prod_{j=1}^\ell \begin{pmatrix}M_j&M_jQ\\0&N_j\end{pmatrix}=
\begin{pmatrix}\prod_{j=1}^\ell M_j&
\sum_{k=1}^\ell \left(\prod_{j=1}^k M_j\right)Q\left(\prod_{j=k+1}^\ell N_j\right)
\\0&\prod_{j=1}^\ell N_j\end{pmatrix}$$
and this can be easily done by induction on $\ell$.

(5) If $w=w_1\cdots w_\ell$, let $M_j=A^{w_j}$ and $Q=\bb a^{-1}\cc$. The statement is proved by induction on $\ell$. It obviously holds for $\ell=0$, so let $\ell\ge 1$. Note that representation \eqref{a4} yields a series $T$ with
$$[T,w]=-a^{-1}\cc A^{w_1}(I-\bb a^{-1}\cc)\cdots A^{w_{\ell-1}}(I-\bb a^{-1}\cc)
A^{w_\ell}\bb a^{-1}.$$
By \eqref{e:inv} and the inductive step we have
\[ \begin{split}
[S^{-1},w]
&=-\sum_{uv=w,v\neq w} a^{-1}[S,u][T,v] \\
&=-a^{-1}\cc\left(M_1\cdots M_\ell-\sum_{i=1}^{l-1} M_1\cdots M_i Q M_{i+1}(I-Q)\cdots M_{\ell-1}(I-Q)
M_\ell\right)\bb a^{-1} \\
&=-a^{-1}\cc M_1(I-Q)\cdots M_{\ell-1}(I-Q)M_\ell\bb a^{-1} \\
&=[T,w]
\end{split} \]
and thus the statement holds.
\end{proof}

Let $(\cc,A,\bb)$ be a linear representation of dimension $n$. For $N\in\N\cup\{0\}$ we define
$$\cU_N=\{\uu\in \cA^{1\times n}\colon \uu A^w\bb=0\ \forall |w|\le N\}.$$
These are left $\cA$-modules and $\cU_N\supseteq \cU_{N+1}$. Furthermore, let
$$\cU_\infty=\bigcap_{N\in\N}\cU_N.$$
In the language of control theory, this module represents an obstruction for the {\it controllability} of the realization \cite[Section 5]{BGM} of a rational function defined in 0.

\begin{lem}\label{l:stop}
If $\cU_N=\cU_{N+1}$, then $\cU_{N+1}=\cU_{N+2}$.
\end{lem}

\begin{proof}
If $\uu\notin \cU_{N+2}$, then $\uu A^{Y_iw} \bb \neq0$ for some $Y_i\in \y$ and $|w|=k\le N+1$. Let $f\in\cA^{Y_iw}$ be the nonzero entry of $\uu A^{Y_iw} \bb$. Let $\{Z_0,\dots,Z_k\}$ be auxiliary freely noncommuting letters; since $\cA^{Y_iw}$ as a $\cA$-module depends only on length of $Y_iw$, we can treat $f$ as an element of $\cA^{Z_0\cdots Z_k}$ and 
$f(Z_0,\dots,Z_k) \neq0$. By Lemma \ref{l:difvar}, there exists $b\in \cA$ such that $f(b,Z_1,\dots,Z_k) \neq0$. Going back to module $\cA^w\cong \cA^{Z_1\cdots Z_k}$, we have 
$f(b,w_1,\dots,w_k)\neq0$ and so $\uu A^{Y_i}|_{Y_i=b}A^w\bb\neq0$. Therefore 
$\uu A^{Y_i}|_{Y_i=b}\notin \cU_{N+1}=\cU_N$ and hence 
$\uu A^{Y_i}|_{Y_i=b}A^{w'}\bb\neq0$ for some 
$|w'|=k'\le N$. Let $g\in\cA^{Z_0w'}$ be such entry of 
$\uu A^{Y_i}|_{Y_i=Z_0}A^{w'}\bb$ that $g(b,w'_1,\dots,w'_{k'})\neq0$. Thus 
$g\neq0$ and also $g(Y_i,w'_1,\dots,w'_{k'})\neq0$. Therefore $\uu A^{Y_i}A^{w'}\bb\neq0$ 
and $\uu\notin \cU_{N+1}$.
\end{proof}

\begin{lem}\label{l:termination}
If a representation $(\cc,A,\bb)$ is of dimension $n$, then $\cU_\infty=\cU_{mn-1}$, where $m^2=\dim_{\kk}\cA$.
\end{lem}

\begin{proof}
The statement trivially holds for $\bb=0$, so we assume $\bb\neq0$. By Morita equivalence between $M_m(\kk)$ and $\kk$, the dimension of every left $\cA$-module as a vector space over 
$\kk$ is divisible by $m$. Since the dimension of the vector space $\cA^{1\times n}$ over $k$ is $m^2n$ and $\bb\neq0$, the descending chain of left $\cA$-modules $\{\cU_N\}_{N\in\N}$ stops by Lemma \ref{l:stop} and
\[\cA^{1\times n}\supseteq \cU_0\supseteq \cU_1\supseteq \cdots \supseteq \cU_{mn-1}
=\cU_{mn}=\cdots. \qedhere \]
\end{proof}

\begin{prop}\label{p:RSIbound}
If $(\cc,A,\bb)$ is a representation of dimension $n$ and $\cc A^w\bb=0$ for all $|w|<mn$, then 
$(\cc,A,\bb)$ represents the zero series.
\end{prop}

\begin{proof}
The assumption asserts $\cc\in\cU_{mn-1}$, so $\cc A^w\bb=0$ for all $w\in\gp$ by 
Lemma \ref{l:termination}.
\end{proof}

Let $r$ be a rational expression in $\x$ and assume it is defined in $\pp\in M_m(\kk)^g$. Then $r$ can be formally expanded into a generalized series about $\pp$; more precisely, $r(\y+\pp)$ can be considered as an element of $\ags$. Since $r(\y+\pp)$ lies in the rational closure of $\agp$, it is a recognizable series by Theorem \ref{t:arithmetic}. We say that a linear representation $(\cc,A,\bb)$ is a \textbf{realization of $r$ about $\pp$} if $(\cc,A,\bb)$ is a representation of $r(\y+\pp)$.

\begin{exa}\label{ex:real}\hfill\\[-4ex]
\begin{enumerate}[\rm(1)]

\item Let $r=X_1^{-1}(1-\sum_{j=2}^g X_j Y_j)$. The right-hand side expression is defined in the scalar point $(1,0,\dots,0)$; so $r=S(X_1-1,X_2,Y_2,\dots, X_g,Y_g)$ for $S=(Y+1)^{-1}(1-\sum_{j=2}^g X_jY_j)$ and the latter can be expanded into a (generalized) series. If $g=1$, then $S$ has a linear representation $(1,-Y,1)$ of dimension 1. Otherwise if $g\ge2$, then one can easily check that the inverse of
$$
\begin{bmatrix}
1+Y & 0 & X_2 & \cdots & X_g \\
0     & 1 &  0  & \cdots &  0  \\
0    &-Y_2&  1  & \ \cdots  &  0  \\
\vdots&\vdots& \vdots & \ddots & \vdots \\
0    &-Y_g&  0  & \cdots &  1
\end{bmatrix}
$$
equals
$$
\begin{bmatrix}
(1+Y)^{-1} & -(1+Y)^{-1}\sum_{j>1}X_jY_j & -(1+Y)^{-1}X_2 &\cdots & -(1+Y)^{-1}X_g \\
0     & 1 &  0  &  \cdots &  0  \\
0    &Y_2& 1 &\ \cdots &  0  \\
\vdots&\vdots& \vdots &\ddots & \vdots \\
0    &Y_g&  0 &\cdots &  1
\end{bmatrix}
$$
so $S$ has a representation
$$\left(
\begin{bmatrix} 1&0&0&\cdots&0\end{bmatrix},\ \ \ 
-Y E_{11},\ -X_2 E_{13},\ Y_2 E_{32},\ \dots,\ -X_g E_{1,g+1},\ Y_g E_{g+1,2},\ \ \ 
\begin{bmatrix} 1\\1\\0\\\vdots\\0\end{bmatrix}
\right)$$
of dimension $g+1$, where $E_{ij}\in \kk^{(g+1)\times (g+1)}$ are the standard matrix units.

\item Next consider $r=(X_1X_2-X_2X_1)^{-1}$. Let $\pp=(P_1,P_2)$ be a pair of $2\times2$ matrices such that $Q=(P_1P_2-P_2P_1)^{-1}$ exists, e.g. $P_1=E_{12}$ and $P_2=E_{21}$. Then
$r(X_1,X_2)=S(X_1-P_1,X_2-P_2)$, where
$$S=Q(1-(P_2Y_1-Y_1P_2)Q-(Y_2P_1-P_1Y_2)Q-(Y_2Y_1-Y_1Y_2)Q)^{-1}.$$
Using the blockwise inversion formula (see e.g. \cite[Subsection 0.7.3]{HJ}) it can be easily seen that $S$ has a representation
$$\left(
\begin{pmatrix}Q &0&0\end{pmatrix},\ 
\begin{pmatrix}
-Y_1P_2Q+P_2Y_1Q &Y_1 &0 \\
0 &0 &0 \\
-Y_1Q &0 &0
\end{pmatrix},\ 
\begin{pmatrix}
Y_2P_1Q-P_1Y_2Q &0 &-Y_2 \\
-Y_2Q &0 &0 \\
0 &0 &0
\end{pmatrix}, \
\begin{pmatrix}1\\0\\0\end{pmatrix}
\right)$$
of dimension 3.

\end{enumerate}
\end{exa}

We can now give the main result of this subsection, namely explicit size bounds required for testing whether a nc rational expression is a rational identity.

\begin{thm}\label{t:RIbound}
Let $r$ be a rational expression in $\x$. Assume $r$ admits a realization of dimension $n$ about a point in $M_m(\kk)^g$. If $r$ is an identity on matrices of size 
$N=m\lceil\tfrac{mn}{2}\rceil$, then $r$ is a rational identity.
\end{thm}

\begin{proof}
Let $\pp\in M_m(\kk)^g\cap \dom r$ and $(\cc,A,\bb)$ be a realization of $r$ about $\pp$ of dimension $n$. Let $\x(N)$ be the tuple of generic $N\times N$ matrices, i.e. 
$X_j(N)=(x_{\imath\jmath}^{(j)})_{\imath\jmath}$, 
where $x_{\imath\jmath}^{(j)}$ are independent commuting variables. Because $M_m(\kk)^g\cap \dom r\neq \emptyset$ and $N$ is a multiple of $m$, we also have $M_N(\kk)^g\cap \dom r\neq \emptyset$, so $r$ can be evaluated on $\x(N)$. Then $r(\x(N))$ is a matrix of commutative rational functions and the matrix of their expansions about $\pp$ is
\begin{equation}
M(\x(N))=\sum_w\cc\left(A^w|_{\y=\x(N)-\pp}\right)\bb.
\end{equation}
The formal differentiation of these commutative power series yields
\begin{equation}\label{der}
\frac{\der}{\der t^h} M\big(\pp+t(\x(N)-\pp)\big)\Big|_{t=0}
=h!\sum_{|w|=h}\cc \left(A^w|_{\y=\x(N)-\pp}\right)\bb.
\end{equation}
If $r$ vanishes on matrices of size $N$, then $r(\x(N))=0$ and so $M(\x(N))=0$; therefore the left-hand the side of \eqref{der} equals 0 for every $h$, hence the same holds for the right-hand side. Since $\sum_{|w|=h}\cc (A^w|_{\y=\x-\pp})\bb$ is a generalized polynomial of degree $h$, we have $\cc A^w\bb=0$ for all $|w|<mn$ by Proposition \ref{p:GPI}. Finally $r$ is a rational identity by Proposition \ref{p:RSIbound}.
\end{proof}

\subsection{Bounds for grr ideals}\label{s:bounds}

We now have enough tools at our disposal to prove the main result of this section.

\begin{thm}\label{t:rrbounds}
Let $\cI$ be grr ideal with rational resolvent $r=(r_1,\dots, r_k)$. Assume there is 
a tuple of $m\times m$ matrices $\pp\in \dom r$ and that rational functions $r_j$ can be defined by rational expressions with realizations about $\pp$ of dimensions at most $n$.

If $f\in \kk\ax$ is of degree $u$ and has $v$ terms, and vanishes on $Z(\cI)\cap M_N(\kk)^g$, where
\beq\label{eq:N}
N=m\lceil\tfrac{m u v \max(n,2)}{2}\rceil,
\eeq
then $f$ vanishes on $Z(\cI)$.
\end{thm}

\begin{proof}
As already observed at the beginning of this section, it is enough to prove that $q=f(\x',r(\x'))$ is a rational identity. By the assumptions and Theorem \ref{t:arithmetic}, this rational function can be defined by a rational expression with realization about $\pp$ of dimension at most $uv\max(n,2)$. Indeed, when constructing the realization of $q$ using Theorem \ref{t:arithmetic} from realizations of $r_j$, every symbol in $f$ contributes a realization of dimension either 2 (if it belongs to $\x'$) or at most $n$ (if it belongs to $\x''$), and the sums and products result in addition of the dimensions of intermediate realizations. Now the statement follows by Theorem \ref{t:RIbound}.
\end{proof}

\begin{cor}
Assume the setting of Theorem {\rm\ref{t:rrbounds}}. If $f\in \kk\ax$ is of degree $d$ and vanishes on $Z(\cI)\cap M_{N'}(\kk)^g$, where
$$N'=m\lceil\tfrac{m d(g+1)^d \max(n,2)}{2}\rceil,$$
then $f$ vanishes on $Z(\cI)$.
\end{cor}

Let us determine the bound $N$ 
in \eqref{eq:N}
from Theorem \ref{t:rrbounds} as a function of $uv$ for some concrete ideals.

\begin{exa}\label{exa:bounds}
\hfill\\[-4ex]
\begin{enumerate}[\rm(1)]

\item The ideal $\cT'=(1-X_1Y_1,1-Y_1X_1,\ldots,1-X_gY_g,1-Y_gX_g)$ is a special case of an ideal from Corollary \ref{c:exfrr1}, so it is grr. Its rational resolvent consists of the functions $X_j^{-1}$, which have realizations of dimension 1 about the point $(1,\dots, 1)$ by Example \ref{ex:real}(1). Thus Theorem \ref{t:rrbounds} implies $N=uv$.

\item Let $g\ge2$; the ideal $\cS'=(X_1Y_1+\cdots+X_gY_g-1)$ was studied in Corollary \ref{c:exfrr2} and is grr. By Example \ref{ex:real}(1), its resolvent has a realization of dimension $g+1$ about a scalar point, so $N=\lceil\tfrac{g+1}{2}uv\rceil$.

\item Consider the ideal $\cI=(1-(X_1X_2-X_2X_1)X_3)\subset \kk\ax$; evidently it is grr with rational resolvent $r=(X_1X_2-X_2X_1)^{-1}$. Hence $N=6uv$ by Example \ref{ex:real}(2).

\item Lastly, let $X=(X_{ij})_{ij}$ and $Y=(Y_{ij})_{ij}$ be $g\times g$ matrices with freely noncommuting entries and let $\cU'$ be the ideal generated by the entries of $XY-I_g$ and $YX-I_g$. Assume  $g>1$; the case $g=1$ is treated in (1). As already observed in the proof of Corollary \ref{c:exfrr1}, $\cU'$ is grr, with the resolvent consisting of the entries of $X^{-1}$. Since the $(i,j)$-th entry of this matrix equals
$$e_i^t (I_g-(-X_g+I_g))^{-1}e_j,$$
where $e_i$ and $e_j$ are the standard unit vectors in $\kk^g$, it has a realization of dimension $g$ about $I_g$. Therefore Theorem 
\ref{t:rrbounds} yields $N=\lceil\tfrac{g}{2}uv\rceil$.

\end{enumerate}

\end{exa}

\section{Null- and Positivstellens\"atze for $*$-ideals} \label{sec4}

In this section we turn our attention to algebras with involution. In addition to zero sets this setting also rises questions about positivity sets of nc polynomials \cite{HMP2}. We give Null- and Positivstellens\"atze for certain classes of rationally resolvable $*$-ideals in free $*$-algebras. We prove a Nullstellensatz for nc unitary groups and spherical isometries 
(see Theorems  \ref{thm:null3} and \ref{thm:null2}) and use it to deduce
Positivstellens\"atze in Subsection \ref{subsec:pos},
following  work of Helton, McCullough and Putinar \cite{HMP1}.
To pass between our results
for free algebras and free $*$-algebras we employ (real) algebraic geometry, cf.~Subsection \ref{subs:real}.

We shall be interested in
the free $*$-algebra. 
Let $\axs$ be the monoid freely
generated by $\x=\{X_1,\ldots, X_g\}$ and 
$\xs=\{X_1^\ss,\ldots,X_g^\ss\}$, i.e., $\axs$ consists of  words in the $2g$
noncommuting
letters $X_{1},\ldots,X_{g},X_1^\ss, \ldots,X_g^\ss$
(including the empty word $\emptyset$ which plays the role of the identity $1$).
For a field $\kk$ endowed with an involution (an automorphism of order $2$)
let $\kk\axs$ denote the $\kk$-algebra freely generated by $\x,\xs$. This is a free algebra with involution ${}^\ss$ that is uniquely determined by the involution of the base field and the rule $X_j^{\ss\ss}=X_j$. An ideal $\cI\subset\kk\axs$ is called a {\bf $*$-ideal} if $\cI^{\ss}=\cI$.

\subsection{\except{toc}{Embedding quotients by rationally resolvable $*$-ideals into skew fields with involution} \for{toc}{Embedding $*$-quotients into skew fields with involution}}

Let $\cI$ be a frr $*$-ideal in $\fr{\x,\x^\ss}$ with rational resolvent $r$. Because we are now dealing with two different partitions of nc variables, namely $\x\cup \x^\ss=\x'\cup \x''$, where the right-hand side comes from the decomposition \eqref{decom}, we introduce some additional notation:
$$\x^{(1)}=\x'\cap \x'^\ss,\quad \x^{(2)}=\x'\setminus \x^{(1)},\quad \x^{(3)}=\x^{(2)\ss},\quad \x^{(4)}=\x''\setminus \x^{(3)}.$$
For example, if $g=3$, $X'=\{X_1,X_1^\ss,X_2\}$ and $X''=\{X_2^\ss,X_3,X_3^\ss\}$, then $\x^{(1)}=\{X_1,X_1^\ss\}$, $\x^{(2)}=\{X_2\}$, $\x^{(3)}=\{X_2^\ss\}$ and $\x^{(4)}=\{X_3,X_3^\ss\}$. 
Some caution is needed when considering these partitions as arguments in an expression. If, for example, $s=s(U,V)$ is an expression in two variables, we write $s(\x^{(1)})=s(X_1,X_1^\ss)$ and 
$s(\x^{(1)\ss})=s(X_1^\ss,X_1)$, because here $\x^{(1)}$ and $\x^{(1)\ss}$ are different as lists of arguments, although they are equal as sets.

Let $r_\bullet$ be the subtuple of $r$ corresponding to $\x^{(3)}$.

\begin{lem}\label{lem1}
If the notation is as above, then
\beq\label{inv}
\x^{(2)}=r_\bullet^\ss\left(\x^{(1)\ss},r_\bullet(\x')\right)
\eeq
holds in $\ff{\x'}$.
\end{lem}

\begin{proof} By definition, $\cI_r=R_r\cI R_r$ holds and therefore
$$\x^{(3)}-r_\bullet\left(\x^{(1)},\x^{(2)}\right)\in R_r\cI R_r.$$
Since $\cI$ is closed under involution, it also follows that
$$\x^{(2)}-r_\bullet^\ss\left(\x^{(1)\ss},\x^{(3)}\right)\in R_{r^\ss}\cI R_{r^\ss},$$
where $R_{r^\ss}$ is the subring generated by $\ff{\x'^\ss}$ and $\fr{\x,\x^\ss}$. Combining these two results yields
$$\x^{(2)}-r_\bullet^\ss\left(\x^{(1)\ss},r_\bullet\left(\x^{(1)},\x^{(2)}\right)+R_r\cI R_r\right)
\in R_{r^\ss}\cI R_{r^\ss}.$$
Since $\Gamma(r)\subseteq Z(\cI)$, substituting $\x'$ and $\x''$ by $\ua$ and $r(\ua)$ in these expressions, respectively, we have
$$\ua^{(2)}-r_\bullet^\ss\left(\ua^{(1)\ss},r_\bullet(\ua)\right)=0$$
for all tuples $\ua$ in the intersection of domains of all rational functions that appear in the upper expressions. Since this set can be again realized as a domain of a rational function, the considered equalities give rise to rational identities by \cite[Theorem 16]{Ami}.
\end{proof}

\begin{prop}\label{prop3}
If a $*$-ideal $\cI$ in $\fr{\x,\x^\ss}$ satisfies the assumptions of Theorem {\rm\ref{thm1}(a)}, then $\fr{\x,\x^\ss}\!/\cI$ $*$-embeds into a free skew field with an involution.
\end{prop}

\begin{proof} By Theorem \ref{thm1}, the homomorphism
$$\Phi: \fr{\x,\x^\ss}\to \ff{\x'},\quad p\mapsto p(\x',r(\x'))$$
induces an embedding $\fr{\x}\!/\cI \hookrightarrow \ff{\x'}$. Define an antilinear antihomomorphism of $\kk$-algebras $i:\fr{\x'}\to\ff{\x'}$ by setting
$$i(\x^{(1)})=\x^{(1)\ss},\quad i(\x^{(2)})=r_\bullet(\x').$$
By the universal property of $\ff{\x'}$ as stated in \cite[Section 4.4]{Coh1}, there exists a local antilinear antihomomorphism $ \ff{\x'}\supseteq K \to \ff{\x'}$
which we also denote $i$. By Lemma \ref{lem1}, $i(\fr{\x'})\subseteq K$, so there is a homomorphism $j:\fr{\x'}\to \ff{\x'}$ defined as $j(p)=i(i(p))$. By the same argument as above, $j$ extends to a local homomorphism of free skew fields with domain $L\subseteq K$. Since $j=\id_L$ holds by \eqref{inv} in Lemma \ref{lem1}, $j$ is injective and therefore $\ff{\x'}=L=K$ by the definition of a local homomorphism. Therefore $i$ is an involution of the free skew field $\ff{\x'}$. Now the claim follows since $\Phi$ is compatible with $i$ and the involution on $\fr{\x,\x^\ss}$.
\end{proof}

\subsection{Examples}

In this short subsection we present the main examples of interest to us: nc trigonometric and spherical polynomials, as well as nc unitary groups.

\subsubsection{$*$-representations}

From here on let $\kk=\C$. If $p\in\C\axs$ is an nc polynomial and $\ua\in\ M_n(\C)^{g}$, the evaluation $p(\ua)\in M_n(\C)$ is defined by simply replacing $X_{i}$ by $A_{i}$ and $X_i^\ss$ by $A_i^*$, where $*$ is the conjugate transposition. These polynomial evaluations give rise to finite-dimensional $*$-representations of nc polynomials. The notion of a zero set of a $*$-ideal translates accordingly:
$$Z_*(\cI):= \bigcup_{n\in\N} \{ \ua\in M_n(\C)^g \mid
\forall g\in\cI: \,g(\ua,\ua^*)=0 \}.$$

\subsubsection{nc trigonometric polynomials}

Let 
\beq\label{eq:nctrig}
\cT= (1-X_1^{\ss}X_1,\, 1-X_1X_1^{\ss},\,\ldots,\, 1-X_g^{\ss}X_g,\,1-X_gX_g^{\ss})
\eeq
be a $*$-ideal of $\C\axs$.
The quotient $\C\axs\!/\cT$ is called the algebra of {\bf nc trigonometric
polynomials}. Obvioulsy it is isomorphic to the group algebra of the free group on $g$ letters. We are interested in finite-dimensional $*$-representations
of $\C\axs\!/\cT$, i.e., we evaluate $p\in\C\axs$ at $g$-tuples
consisting of unitaries $U_j$.

\subsubsection{nc spherical polynomials}

Let 
\beq\label{eq:ncspherical}
\cS= (1-X_1^{\ss}X_1- \cdots-X_g^{\ss}X_g)
\eeq
be a $*$-ideal of $\C\axs$.
The quotient $\C\axs\!/\cS$ is called the algebra of {\bf nc spherical
polynomials}.
Here we consider evaluations of 
$p\in\C\axs$ at $g$-tuples
of {\bf spherical isometries} $\ua$.

\subsubsection{nc unitary groups}

Let $X=(X_{ij})_{ij}$ be a $g\times g$ matrix of freely noncommuting symbols and let $\cU$ be the ideal in
$$\C\axs=\frc{X_{ij},X_{ij}^{\ss}\colon 1\le i,j\le g}$$
generated by the set of $2g^2$ relations imposed by $X X^\ss=I_g$ and $X^\ss X=I_g$. Then the algebra $\C\axs\!/\cU$ is a {\bf nc unitary group}. The notion is due to Brown \cite{Bro}; see also \cite{Wo}. The points of the corresponding zero set are $g^2$-tuples $\ua=(A_{ij})_{i,j=1}^g$ of square matrices of the same size satisfying $\ua^* \ua=I_g$. We say that such tuple is a 
\textbf{$g$-partitioned unitary}. Matrices $A_{ij}$ are called blocks of $\ua$.

\subsection{Real structure on a complex variety}\label{subs:real}

The aim of this subsection is to establish a few assertions from algebraic geometry that will enable us to use the results of Section \ref{sec2} in the involution setting. By a {\bf variety} we always mean a Zariski closed subset of an affine space.

Let $V$ be a $\C$-vector space. A map $J:V\to V$ is a {\bf real structure on $V$} if it is conjugate-linear and satisfies $J^2=\id_V$. If $V_J$ is the $J$-fixed subspace of $V$, then $\dim_\R V_J=\dim_\C V$. If $\cX\subset V$ is a $\C$-variety, then $J$ is a real structure on $\cX$ if $J (\cX)\subset \cX$. In this case there is a corresponding conjugate-linear homomorphism $J^*:\C[\cX]\to \C[\cX]$ of coordinate rings. Moreover, we get a real structure $J_x$ on the tangent space $\Theta_{\cX,x}$ of $\cX$ at $x$ for any $x\in \cX$. Let $\cX_J$ be the $J$-fixed subset of $\cX$; this is a $\R$-variety. It is not hard to see that $\Theta_{\cX_J,x}=(\Theta_{\cX,x})_{J_x}$ for $x\in \cX_J$.

The following proposition is well-known (see e.g.~\cite[Lemma 1.5]{Be} or \cite[Theorem 4.9]{DE70} for stronger versions), but we provide a short proof for the sake of completeness.

\begin{prop}\label{p:real}
Let $\cX$ be an irreducible $\C$-variety with a real structure $J$ and assume there exists $x\in \cX_J$ which is a nonsingular point of $\cX$. Then $\cX_J$ is Zariski dense in $\cX$.
\end{prop}

\begin{proof}
Let $\dim_\C \cX=N$. Since $x$ is nonsingular, we have
$$N=\dim_\C\Theta_{\cX,x}=\dim_\R(\Theta_{\cX,x})_{J_x}=\dim_\R\Theta_{\cX_J,x}.$$
If $\cX_J\subseteq \cX'$ for some complex subvariety $\cX'\subseteq \cX$, then $\Theta_{\cX_J,x}+i\Theta_{\cX_J,x}\subseteq \Theta_{\cX',x}$, so $\dim_\C \Theta_{\cX',x}=N$ and therefore $\cX'=\cX$ by irreducibility. Hence $\cX_J$ is Zariski dense in $\cX$.
\end{proof}

For later use we introduce
\beq\label{eq:Xgn}
\cX(g,n)=\left\{(\ua,\ub)\in M_n(\C)^g\times M_n(\C)^g
\colon \sum_{k=1}^g A_k B_k=I_n\right\}
\eeq
for arbitrary $g,n\in\N$. This is a $\C$-variety with real structure 
$J(\ua,\ub)=(\ub^*,\ua^*)$ and
\[
\cX(g,n)_J=\left\{(\ua,\ua^*)\in M_n(\C)^g\times M_n(\C)^g
\colon \sum_{k=1}^g A_k A_k^*=I_n\right\}.
\]

\begin{prop}\label{p:var}
The variety $\cX(g,n)$ is nonsingular and irreducible. Therefore $\cX(g,n)_J$ is Zariski dense in $\cX(g,n)$.
\end{prop}

\begin{proof}
Let
$$p=\sum_{k=1}^g X_k Y_k-I_n,$$
where $X_k$ and $Y_k$ are generic $n\times n$ matrices. The entries of $p$ are the defining equations for $\cX(g,n)$ and
\begin{equation}\label{e:partial}
\frac{\partial p_{\imath\jmath}}{\partial x_{ij}^{(k)}}=\left\{
\begin{array}{ll}
y_{j\jmath}^{(k)} & \text{if } \imath=i \\
0 & \text{otherwise,}
\end{array}\right. \qquad
\frac{\partial p_{\imath\jmath}}{\partial y_{ij}^{(k)}}=\left\{
\begin{array}{ll}
x_{\imath i}^{(k)} & \text{if } \jmath=j \\
0 & \text{otherwise.}
\end{array}\right.
\end{equation}
Let $\Jac$ be the Jacobian matrix corresponding to $p$, i.e. 
the $n^2\times 2gn^2$ matrix
$$\Jac=\begin{pmatrix}
\frac{\partial p_{\imath\jmath}}{\partial x_{ij}^{(k)}} & 
\frac{\partial p_{\imath\jmath}}{\partial y_{ij}^{(k)}}
\end{pmatrix}_{\imath,\jmath;i,j,k}.$$
By \eqref{e:partial}, one can observe that every column of $\Jac$ is of the form $\uu \otimes e_\ell$ or $e_\ell\otimes \vv^t$, where $\uu$ is a column of $X_k$, $\vv$ is a row of $Y_k$, and $e_\ell\in\C^{n\times 1}$ is the $\ell$-th standard unit vector. Therefore
$$\rank \Jac(\ua,\ub)=
n\dim_\C \big(\span\{\text{columns of }\ua, \text{ columns of }\ub^t\}\big)$$
for every $(\ua,\ub)\in M_n(\C)^g\times M_n(\C)^g$. If $(\ua,\ub)\in\cX(g,n)$, then the columns of $\ua$ are linearly independent, so $\Jac(\ua,\ub)$ has full rank. Therefore $\cX(g,n)$ is nonsingular.

For any variety, the intersections of its irreducible components are subsets of the singular locus; thus $\cX(g,n)$ is irreducible if it is connected in Euclidean topology. Since $\GL_n(\C)$ is connected, the same holds for $\cX(1,n)=\{(A,A^{-1})\colon A\in \GL_n(\C)\}$. For arbitrary $g$, there is a surjective projection $\cX(1,gn)\to \cX(g,n)$, so $\cX(g,n)$ is connected.

The last part of the statement is a consequence of Proposition \ref{p:real}.
\end{proof}

\subsection{$*$-Nullstellens\"atze}\label{subs:starnull}

In this subsection we give Nullstellens\"atze for nc trigonometric polynomials, nc spherical polynomials, and nc unitary groups. These are obtained by combining the results of Section \ref{sec2} and Subsection \ref{subs:real}. 
Alternative proofs of these $*$-Nullstellens\"atze (without size bounds) 
with a functional-analytic flavor are presented in Appendix \ref{sec:rfd} below.

Let $\cX(g,n)$ be as in Subsection \ref{subs:real}.

\begin{defn} A $*$-ideal $\cI\subset\C\axs$ satisfies the {\bf $*$-Nullstellensatz property} if
for each $f\in\C\ax$,
$$f\in\cI \quad\Leftrightarrow\quad f|_{Z_*(\cI)}=0.$$
\end{defn}

\begin{thm}\label{thm:null1}
Suppose $p\in\C\axs$ is of degree $u$ and has $v$ terms. If $p$ vanishes on all $g$-tuples of unitaries of size $uv$, then
$$p\in\cT=(1-X_1^{\ss}X_1,\, 1-X_1X_1^{\ss},\,\ldots,\, 1-X_g^{\ss}X_g,\,1-X_gX_g^{\ss}).$$
\end{thm}

\begin{proof}
Let $n=uv$. By assumption, $p$ vanishes on $\cX(1,n)_J\times\cdots \times \cX(1,n)_J$ ($g$ factors), so it vanishes on $\cX(1,n)\times\cdots \times \cX(1,n)$ by Proposition \ref{p:var}. The latter variety is a subset of the zero set of $\cT$ as an ideal without involution. By Corollary \ref{c:exfrr1}, this ideal satisfies the Nullstellensatz property. Therefore $p$ vanishes on $Z(\cT)$ by Example \ref{exa:bounds}(1), so $p\in\cT$.
\end{proof}

\begin{cor}
A nc trigonometric polynomial that vanishes under all finite-dimensional
$*$-representations, equals $0$.
\end{cor}

\begin{thm}\label{thm:null2}
Suppose $p\in\C\axs$ is of degree $u$ and has $v$ terms. Let $g>1$. If $p$ vanishes on all $g$-tuples of spherical isometries of size $\lceil\tfrac{g+1}{2}uv\rceil$, then
$$p\in\cS=(1-X_1^{\ss}X_1- \cdots-X_g^{\ss}X_g).$$
\end{thm}

\begin{proof}
Let $n=\lceil\tfrac{g+1}{2}uv\rceil$. Since $p$ vanishes on $\cX(g,n)_J$, it vanishes on $\cX(g,n)$ by Proposition \ref{p:var}. The ideal $\cS$ satisfies the Nullstellensatz property by Corollary \ref{c:exfrr2}. Hence $p\in\cS$ because $p$ vanishes on the zero set of $\cS$ by Example \ref{exa:bounds}(2).
\end{proof}

\begin{cor}
Let $g>1$. A nc spherical polynomial that vanishes under all finite-dimensional
$*$-representations, equals $0$.
\end{cor}

\begin{thm}\label{thm:null3}
Suppose $p\in\C\axs$ is of degree $u$ and has $v$ terms. Let $g>1$. If $p$ vanishes on all $g$-partitioned unitaries with blocks of size $\lceil\tfrac{g}{2}uv\rceil$, then
$$p\in\cU=(I_g-XX^{\ss}, I_g-X^{\ss}X).$$
\end{thm}

\begin{proof}
The variety of $g$-partitioned unitaries with blocks of size $n$ can be naturally identified with $\cX(1,gn)_J$. On the other hand, $\cX(1,gn)$ can be in the same way identified with the zero set of $\cU$. The latter is a special case of an ideal from Corollary \ref{c:exfrr1}. Therefore $p\in\cU$ by the Nullstellensatz property and Example \ref{exa:bounds}(4).
\end{proof}

\begin{rem}\label{r:general}
In a similar fashion, one can derive the $*$-Nullstellens\"atze for any $*$-ideal as in Corollary \ref{c:exfrr1} by considering the appropriate products of varieties $\cX(g,n)$ for various values of $g$ and $n$. Furthermore, as in Theorems \ref{thm:null1} and \ref{thm:null2}, bounds for membership testing can be established using Theorem \ref{t:rrbounds}.
\end{rem}

\begin{rem}\label{r:realcase}
Similar $*$-Nullstellens\"atze also hold in the real setting, i.e., when we consider  evaluations of polynomials $p\in\R\axs$ at points $\ua\in M_n(\R)^g$ by replacing $X_i$ by $A_i$ and $X_i^\ss$ by $A_i^t$, where $t$ denotes the transposition. Using the natural $*$-embedding
$$M_n(\C)\hookrightarrow M_{2n}(\R),\quad A+iB\mapsto \begin{pmatrix}A&-B\\B&A\end{pmatrix}$$
we see that if $p\in \R\axs$ vanishes on
$$\left\{(\ua,\ua^t)\in M_{2n}(\R)^g\times M_{2n}(\R)^g
\colon \sum_{k=1}^g A_k A_k^t=I_n\right\} \hookleftarrow \cX(g,n)_J,$$
then it vanishes on
$$\left\{(\ua,\ub)\in M_n(\R)^g\times M_n(\R)^g
\colon \sum_{k=1}^g A_k B_k=I_n\right\}\subset \cX(g,n).$$
At this point we can apply results from Sections \ref{sec2} and \ref{sec3} for $\kk=\R$. Thus we obtain the $*$-Nullstellens\"atze for real versions of Theorems \ref{thm:null1}, \ref{thm:null2} and \ref{thm:null3}, but with size bounds multiplied by 2.
\end{rem}

\subsection{Positivstellens\"atze}\label{subsec:pos}

Let $\cZ$ be a set of $g$-tuples of matrices. Then a polynomial $f\in\C\axs$ is said to be \textbf{positive on $\cZ$} if $f(\ua)$ is positive semi-definite for every $\ua\in\cZ$. Obvious representatives of such polynomials are those of the form
\begin{equation}\label{e:psd}
\sum_i p_i^{\ss} p_i+q,
\end{equation}
where $q$ vanishes on $\cZ$. We are interested in sets $\cZ$ for which the converse of this observation holds; that is, is every $f$ positive semi-definite on $\cZ$, of the form \eqref{e:psd}? Such statements are traditionally called Positivstellens\"atze \cite{BCR, Ma, Sch, HM, HMP}. The basic case, where $\cZ$ is the set of all $g$-tuples of matrices of all sizes, was established by Helton \cite{Hel}; see also \cite{McC}. Later on, a Positivstellensatz was also established for spherical isometries and tuples of unitaries
in \cite{HMP1}. In the case of unitaries we also refer to \cite{NT} for a different approach.
Using Theorems \ref{thm:null1} and \ref{thm:null2}, we can now generalize these results in two directions.
First of all, 
we clearly identify the elements of the vanishing ideals, and 
we prove a Positivstellensatz for nc unitary groups.

For $d\in \N$, let $\cP_d\subset \C\axs$ be the subspace of polynomials of degree at most $d$ and
\[\cC_{2d}=\co\{p^{\ss} p\colon p\in\cP_d\}\]
the associated convex cone of  sums of Hermitian squares. 

\begin{cor}\label{c:pos1}
If $f\in\C\axs$ of degree $d-1$ is positive on all $g$-tuples of unitaries of size $(2g+1)^d$, then
$$f\in   
\cC_{2d}
+(1-X_1^{\ss}X_1,\, 1-X_1X_1^{\ss},\,\ldots,\, 1-X_g^{\ss}X_g,\,1-X_gX_g^{\ss}).$$
\end{cor}

\begin{cor}\label{c:pos2}
If $f\in\C\axs$ of degree $d-1$ is positive on all spherical isometries of size $(2g+1)^d$, then
$$f\in   
\cC_{2d}+(1-X_1^{\ss}X_1- \cdots-X_g^{\ss}X_g).$$
\end{cor}
\begin{proof}[Proof of Corollaries {\rm \ref{c:pos1}} and {\rm\ref{c:pos2}}]
Let $\cZ$ be the set of all tuples of unitaries (resp. spherical isometries). If a polynomial $f$ is positive on $\cZ$, then $f$ is of the form \eqref{e:psd} by \cite[Theorem 4.1]{HMP1} and the statement follows by Theorem \ref{thm:null1} (resp. Theorem \ref{thm:null2}).
\end{proof}

Finally, we adapt the proof of \cite{HMP1} to yield a Positivstellensatz 
for nc unitary groups, characterizing all nc polynomials positive on partitioned unitaries.

\begin{thm}\label{thm:pos1}
Let $\cZ$ be the set of $g$-partitioned unitaries with blocks of size $(2g^2+1)^d$. If a polynomial $f$ of degree $d-1$ is positive on $\cZ$, then
$$
f\in   
\cC_{2d}
+(I_g-\x^{\ss}\x,\, I_g-\x\x^{\ss}).$$
\end{thm}

\begin{proof}
Let $I(\cZ)\subset\C\axs$ be the ideal of all polynomials vanishing on $\cZ$, and set
$$\cP_d(\cZ)=\frac{\cP_d+I(\cZ)}{I(\cZ)},\quad 
\cC_{2d}(\cZ)=\frac{\cC_{2d}+I(\cZ)}{I(\cZ)}.$$
We write $[p]$ to indicate the class of $p$ in these quotients. For the following proof to work, it is crucial that $\cC_{2d}(\cZ)$ is closed in $\cP_{2d}(\cZ)$ \cite[Lemma 3.2]{HMP1}.

By this notation, we have $f\in\cP_{d-1}$ and it is enough to show $[f]\in \cC_{2d}(\cZ)$ by Theorem \ref{thm:null3}.

Let $f\in \cP_{d-1}$ and suppose $f$ is positive on $\cZ$ but $[f]\notin \cC_{2d}(\cZ)$. As was shown in \cite{HMP1}, there is a linear functional $L$ on $\cP_{2d}(\cZ)$ such that $L(f)<0$ and $L(c)>0$ for $c\in \cC_{2d}(\cZ)\setminus\{0\}$. Let $\Lambda$ be the functional on $\cP_{2d}$ obtained by pulling back $L$. Then
$$\langle a,b \rangle=\frac{1}{2}\Lambda(a^{\ss} b+b^{\ss} a)$$
is a Hermitian positive semi-definite form on $\cP_d$ and its associated Hilbert space is $\cH=\cP_d(\cZ)$. Note that $\dim\cH\le\dim\cP_d\le (2g^2+1)^d$, so $f$ is positive on $g$-partitioned unitaries whose blocks are operators on $\cH$. Let $\cM=\cP_{d-1}(\cZ)$ be its subspace and $\cN$ the orthogonal complement of $\cM$ in $\cH$.

For $1\le i,j\le g$, define $X_{ij},Y_{ij}:\cM\to\cH$ by $X_{ij}[p]=[x_{ij}p]$ and 
$Y_{ij}[p]=[x_{ij}^{\ss} p]$. Then
$$\langle X_{ij}a,b\rangle=\langle a,Y_{ij}b\rangle$$
for $a,b\in\cM$. Thus the restriction of $X_{ij}^*$ to $\cM$ is $Y_{ij}$; i.e., $X_{ij}^*$ on $\cM$ is multiplication by $x_{ij}^{\ss}$.

Let
$$U_j: \cM\to \bigoplus_{i=1}^g\cH,\qquad
U_j=\bigoplus_{i=1}^g X_{ij}.$$
Considering the inner product on $\bigoplus \cH$, we have
$$\langle U_ia,U_jb\rangle =\delta_{ij}\langle a,b\rangle$$
for all $a,b\in\cM$, where $\delta_{ij}$ is the Kronecker's delta. Therefore $\{U_j\}_j$ are pairwise orthogonal isometries. Since
$$\dim \left(\sum_{j=1}^g \im U_j\right) = g \dim \cM = g (\dim\cH-\dim \cN),$$
there exist pairwise orthogonal subspaces $\cN_1,\dots,\cN_g \subset \bigoplus\cH$ of the same dimension $\dim \cN_j=\dim\cN$ such that
$$\sum_{j=1}^g \cN_j=\left(\sum_{j=1}^g \im U_j\right)^\bot.$$
We can choose isometric isomorphisms $\cN\cong \cN_j$ and extend $U_j$ to $V_j:\cH\to\bigoplus\cH$ by
$$V_j|_{\cN}:\cN\to \cN_j\hookrightarrow \bigoplus_{i=1}^g\cH.$$
These maps are also pairwise orthogonal isometries. If $A_{ij}=(V_j)_i$, then the $g^2$-tuple $\ua$ satisfies 
$\ua^* \ua=I_g$ and is therefore a $g$-partitioned unitary whose blocks are operators on $\cH$.

As in \cite{HMP1}, one can verify that the restrictions of $A_{ij}$ to $\cM$ are $X_{ij}$ and that the compressions of $A_{ij}^*$ to $\cM$ are $Y_{ij}$. The rest also follows as in \cite{HMP1}: since $f$ is of degree $d-1$, $f(\ua)[1]=f(\x)[1]=[f]$. By the assumption $f(\ub)$ is positive semi-definite for $\ub\in\cZ$, so $[f]=[f^{\ss}]$. But then
$$\langle f(\ua)[1],[1]\rangle=\langle [f],[1]\rangle
=\frac{1}{2}(\Lambda(f)+\Lambda(f^{\ss}))=\Lambda(f)<0,$$
a contradiction.
\end{proof}

\begin{rem}
As in Remark \ref{r:realcase} we also obtain Positivstellens\"atze in the real setting that are analogous to Corollaries \ref{c:pos1}, \ref{c:pos2} and Theorem \ref{thm:pos1}.
\end{rem}

\begin{appendices}

\section{Alternative proof of $*$-Nullstellens\"atze}\label{sec:rfd}

In this appendix we give alternative independent proofs of the $*$-Nullstellens\"atze of Subsection 
\ref{subs:starnull}. These proofs are inspired by functional analytic ideas and do not
yield size bounds.

We say that $\C$-algebra with involution $\cA$ is {\bf residually finite-dimensional (rfd)} if it has a separating family of finite-dimensional $*$-representations, that is, for every $a\in\cA\setminus\{0\}$ there exists $n\in \N$ and a $*$-homomorphism $\varphi:\cA\to M_n(\C)$ with $\varphi(a)\neq0$. Here the involution on $M_n(\C)$ is given by conjugate transpose. Thus a $*$-ideal $\cI\subset\frc{\x,\x^\ss}$ satisfies the $*$-Nullstellensatz property if and only if $\frc{\x,\x^\ss}\!/\cI$ is rfd.

\begin{lem}\label{lem2}
Let $\cA$ be a rfd $\C$-algebra. If a set $\{a_1,\dots,a_n\}\subset \cA$ is linearly independent, then there exists a finite-dimensional $*$-representation $\pi$ of $\cA$ such that $\{\pi(a_1),\dots,\pi(a_n)\}$ is linearly independent.
\end{lem}

\begin{proof} Assuming this statement is not true, choose $n$ minimal such that the claim is false; then clearly $n>1$. Therefore for every finite-dimensional $*$-representation $\pi$ of $\cA$ there exist $\lambda_i^\pi\in \C$, not all zero, such that
$$\sum_i\lambda_i^\pi\pi(a_i)=0.$$
By the minimality, there exists a finite-dimensional $*$-representation $\rho$ of $\cA$ such that $\{\pi(a_2),\dots,\pi(a_n)\}$ is linearly independent and
$$\rho(a_1)+\sum_{i>1}\mu_i\rho(a_i)=0$$
for some $\mu_i\in \C$. Considering the $*$-representation $\pi\oplus\rho$, we have
$$\sum_i\lambda_i^{\pi\oplus\rho}\pi(a_i)=0,\quad \sum_i\lambda_i^{\pi\oplus\rho}\rho(a_i)=0$$
and $\lambda_1^{\pi\oplus\rho}\neq0$ for every $\pi$. Since
$$\sum_{i>1}(\lambda_1^{\pi\oplus\rho}\mu_i-\lambda_i^{\pi\oplus\rho})\rho(a_i)=0,$$
we conclude $\lambda_i^{\pi\oplus\rho}=\lambda_1^{\pi\oplus\rho}\mu_i$. Therefore
$$\sum_i\mu_i\pi(a_i)=0$$
holds for all $*$-representations $\pi$. Hence $\sum_i\mu_ia_i=0$ by assumption, so $\{a_1,\dots,a_n\}$ is linearly dependent in $\cA$, a contradiction.
\end{proof}

\begin{prop}\label{prop4}
If $\cA_1$ and $\cA_2$ are rfd $\C$-algebras, then $\cA_1*_{\C} \cA_2$ is rfd $\C$-algebra.
\end{prop}

\begin{proof} Let $S_i$ be an arbitrary finite linearly independent subset of $\cA_i$. For $n\in \N$ let $S(n,S_1,S_2)$ be a subset of $\cA_1*_{\C} \cA_2$, whose elements are words over $S_1\cup S_2$ of length at most $n$ that do not contain two consecutive elements from $S_1$ or $S_2$. By Lemma \ref{lem2}, there exists a finite-dimensional $*$-representation $\pi_i:\cA_i\to M_{d_i}(\C)$ such that $\pi_i(S_i)$ is linearly independent set. By the universal property of the free product in the category of $\C$-algebras with involution, there exists a $*$-homomorphism 
$$\pi=\pi_1*_{\C}\pi_2:\cA_1*_{\C} \cA_2\to M_{d_1}(\C)*_{\C} M_{d_2}(\C);$$
the set $\pi(S(n,S_1,S_2))$ is linearly independent by construction. Since every element of 
$\cA_1*_{\C} \cA_2$ lies in the linear span of some set $S(n,A_1,A_2)$, it suffices to prove that $M_{d_1}(\C)*_{\C} M_{d_2}(\C)$ is rfd.

By \cite[Proposition 2.3]{Avi}, $M_{d_1}(\C)*_{\C} M_{d_2}(\C)$ $*$-embeds into a C*-algebra of linear operators on a Hilbert space, therefore it also $*$-embeds into a free product of $M_{d_1}(\C)$ and $M_{d_2}(\C)$ in the category of C*-algebras; the latter is rfd by \cite[Theorem 3.2]{EL}, so the assertion holds.
\end{proof}

\begin{cor}\label{cor0}
Every finite free product of nc unitary groups is a rfd algebra.
\end{cor}

\begin{proof}
The algebra $\fr{\x,\x^\ss}\!/\cU$ is rfd by \cite[Theorem 1.2]{GW}. The statement then holds by Proposition \ref{prop4}.
\end{proof}

The Nullstellens\"atze of Remark \ref{r:general} (including Theorem \ref{thm:null1} and Theorem \ref{thm:null3}) can now be viewed as various versions of Corollary \ref{cor0}. However, we do not obtain any bounds by this method. We can also give another proof of Theorem \ref{thm:null2}.

\begin{proof}[Proof of Theorem {\rm\ref{thm:null2}}]
The algebra $\kk\axs\!/\cS$ embeds into $\kk\axs\!/\cU$ under a $*$-homo\-morphism analogous to the embedding in the proof of Corollary \ref{c:exfrr2}. Hence it is rfd by Corollary \ref{cor0}, so $\cS$ satisfies the $*$-Nullstellensatz property.
\end{proof}

\end{appendices}

\linespread{1.05}

\end{document}